\documentclass[11pt,twoside]{article}
\usepackage{amsmath,amssymb,amsthm}
\usepackage{graphicx}
\usepackage{epstopdf}
\usepackage{txfonts}
\usepackage{cite}
\usepackage{caption}
\usepackage{bm}
\allowdisplaybreaks[4]

\newtheorem{thm}{Theorem}[section]
\newtheorem{df}{Definition}[section]
\newtheorem{lm}{Lemma}[section]

\newtheorem{prop}{Proposition}[section]

\newcounter{saveeqn}%

\makeatletter
\evensidemargin
\oddsidemargin
\makeatother

\makeatletter
\@addtoreset{equation}{section}
\makeatother

\title{\Large\bf Dynamics of planar vector fields near a non-smooth equilibrium\thanks{
Supported by NSFC \#11871355 and CSC \#201906240094.
}
}

\author{Tao Li, ~~~~Xingwu Chen\!\!
\footnote{Author to whom any correspondence should be addressed. Email address: xingwu.chen@hotmail.com (X. Chen).}
\\
{\small Department of Mathematics, Sichuan University,}\\
{\small Chengdu, Sichuan 610064, P. R. China}
}
\date{}

\begin{document}
\maketitle

\begin{abstract}
In this paper we contribute to qualitative and geometric analysis of planar piecewise smooth vector fields, which consist of
two smooth vector fields separated by the straight line $y=0$ and sharing the origin as a non-degenerate equilibrium.
In the sense of $\Sigma$-equivalence,
we provide a sufficient condition for   linearization and give phase portraits and normal forms for these linearizable vector fields. This condition is hard to be weakened because there exist vector fields which are not linearizable when this condition is not satisfied. Regarding
perturbations, a necessary and sufficient condition for local $\Sigma$-structural stability is established
when the origin is still an equilibrium of both smooth vector fields under perturbations.
In the opposition to this case, we prove that for any piecewise smooth vector field studied in this paper there is a
limit cycle bifurcating from the origin, and there are some piecewise smooth vector fields such that
for any positive integer $m$ there is a perturbation having exactly $m$ limit cycles bifurcating from the origin. Here $m$ maybe infinity.

\end{abstract}

\baselineskip 15pt
\parskip 10pt
\thispagestyle{empty}
\setcounter{page}{1}

\section{Introduction and statement of the main results}
\setcounter{equation}{0}
\setcounter{lm}{0}
\setcounter{thm}{0}
\setcounter{rmk}{0}
\setcounter{df}{0}
\setcounter{cor}{0}

Let ${\mathcal U}\subset\mathbb{R}^2$ be a bounded open set containing the origin $O$, $\mathfrak{X}$ be the set of all $\mathcal{C}^1$ vector fields defined on ${\mathcal U}$ and endowed with the $\mathcal{C}^1$-topology. We consider the piecewise smooth vector field
\begin{eqnarray}
Z(x, y)=\left\{
\begin{aligned}
  &X(x, y)=(X_1(x, y), X_2(x, y))~~~~&& {\rm if}~~ (x, y)\in\Sigma^+,\\
  &Y(x, y)=(Y_1(x, y), Y_2(x, y))~~~~&& {\rm if}~~ (x, y)\in\Sigma^-,\\
\end{aligned}
\right.
\label{sysp}
\end{eqnarray}
where $X, Y\in \mathfrak{X}$ and
$$\Sigma^+=\{(x, y)\in {\mathcal U}: y>0\}\qquad \Sigma^-=\{(x, y)\in {\mathcal U}: y<0\}.$$
Define $\Omega$ as the set of all $Z(x,y)$ satisfying (\ref{sysp}) and endowed with the product topology.
In past two decades, many researchers shift their interest to the study of piecewise smooth vector fields, because such vector fields are ubiquitous in mechanical engineering \cite{QC, HCJL}, feedback control systems \cite{MD, FGKP}, biological systems \cite{AYX, TSY}, electrical circuits \cite{MD}, etc.

Notice that the piecewise smooth vector field (\ref{sysp}) is not defined on $\Sigma=\{(x, y)\in {\mathcal U}: y=0\}$, called {\it discontinuity line} or {\it switching line}. Denote the vector field on $\Sigma$ by $Z_\Sigma$, which is usually defined by the so-called Filippov convention \cite{AFF}, see Section 2 for a review. Here $Z_\Sigma$ is naturally defined as $X$ or $Y$ if $X(x, y)\equiv Y(x, y)$ for all $(x, y)\in\Sigma$.
The vector field (\ref{sysp}), together with $Z_\Sigma$, are called a {\it Filippov vector field}.
In whole paper, speaking of the vector field $Z\in\Omega$, it always means that $Z=Z_\Sigma$ on $\Sigma$.
A point at which $Z\in\Omega$ vanishes is said to be an {\it equilibrium} or {\it singular point}. Hence, an equilibrium of $Z$ is
an equilibrium of either $X$ in $\Sigma^+$ or $Y$ in $\Sigma^-$ or $Z_\Sigma$ in $\Sigma$. Throughout this paper, we call it
a {\it smooth equilibrium} for the first two cases and a {\it non-smooth equilibrium} for the last case.

Regarding the local dynamics of $Z=(X, Y)\in\Omega$ near a smooth equilibrium, the investigation can be reduced to the local dynamics of the smooth vector field $X$ or $Y$ near this equilibrium and, with the efforts of many researchers, a large number of mature theories and methods have been established (see e.g., \cite{ZZF, B-YAK, JKHale}). Therefore, we focus on the local dynamics for non-smooth equilibria, which is more difficult than the smooth case because most theories and methods for smooth vector fields are no longer valid for non-smooth ones.
Although that, in recent twenty years some excellent results about limit cycle bifurcation, normal form and structurally stability
were given in textbooks \cite{MD, AFF} and
journal papers \cite{YAK, MG1, HZ, CGP, ZK, PGL, TCJLC, TCJ, ZZHFF}.
Let $\Omega_0\subset\Omega$ be the set of all piecewise smooth vector fields satisfying
\begin{eqnarray}
X(0, 0)=Y(0, 0)=(0, 0),~~~~~~\det A^+\det A^-\ne0
\label{adc}
\end{eqnarray}
and
\begin{eqnarray}
X_{2x}(0, 0)Y_{2x}(0, 0)>0,
\label{condi}
\end{eqnarray}
where $A^+$ (resp. $A^-$) is the Jacobian matrix of $X$ (resp. $Y$) at $O$ and
$X_{2x}, Y_{2x}$ denote the derivatives of $X_2, Y_2$ with respect to $x$, respectively.
(\ref{adc}) means that equilibrium $O$ is non-degenerate for both $X$ and $Y$, (\ref{condi}) means that
there exists a hollow neighborhood of $O$, in which there are no sliding points (see Section 2).

In this paper we study the local dynamics of vector field $Z=(X, Y)\in\Omega_0$ near $O$, which
is a non-smooth equilibrium of $Z$, i.e., $Z(0, 0)=Z_\Sigma(0, 0)=(0, 0)$.
Our first goal is to study the {\it local $\Sigma$-equivalence} between $Z=(X, Y)\in\Omega_0$ and its linear part
\begin{eqnarray}
Z_L(x, y)=
\left\{
\begin{aligned}
&X_L(x, y)=A^+(x, y)^\top~~~~~&& {\rm if}~~ (x, y)\in\Sigma^+,\\
&Y_L(x, y)=A^-(x, y)^\top~~~&& {\rm if}~~ (x, y)\in\Sigma^-\\
\end{aligned}
\right.
\label{pwl}
\end{eqnarray}
near $O$. Roughly speaking, the local $\Sigma$-equivalence is just the local topological equivalence preserving the switching line $\Sigma$.
A precise definition of local $\Sigma$-equivalence is stated in Section 2. A nonlocal definition of $\Sigma$-equivalence,
e.g., not in a neighborhood of equilibrium but in the whole domain of definition, was given in \cite[Definition 2.20]{MG1} and \cite[Definition 2.30]{MD}. One of motivations for this goal comes from the work \cite{XCZ}. In \cite[Theorem 2.2]{XCZ}, 19 different types of normal forms for $Z\in\Omega$ with (\ref{adc}) were obtained by using a continuous piecewise linear change of variables. 
We notice that in these normal forms the linear parts are normalized but the nonlinear parts are not normalized.
So, it is unknown that whether these nonlinear parts can be eliminated after normalization.
Another motivation is from smooth vector fields. A smooth vector field is locally topologically equivalent to its linear part near an equilibrium if all eigenvalues of the Jacobian matrix at this equilibrium have nonzero real part (see, e.g., \cite{PH} and \cite[Theorem 4.7]{ZZF}). Hence, it is a natural question to find conditions such that $Z\in\Omega_0$ is locally $\Sigma$-equivalent near $O$ to its linear part $Z_L$ given in (\ref{pwl}).

Let $\lambda^\pm_1$ and $\lambda^\pm_2$ be the eigenvalues of $A^\pm$, and
\begin{eqnarray}
\Omega_1=\{Z\in\Omega_0: \lambda^+_1\ne\lambda^+_2, \lambda^-_1\ne\lambda^-_2, \ell\ne0\},
\label{subsetdefinition}
\end{eqnarray}
where
\begin{eqnarray}
\ell=
\left\{
\begin{aligned}
&\frac{{\rm Re}\lambda^+_1}{|{\rm Im}\lambda^+_1|}+\frac{{\rm Re}\lambda^-_1}{|{\rm Im}\lambda^-_1|}~~~
 &&{\rm if}~~~~~{\rm Im}\lambda^+_1{\rm Im}\lambda^-_1\ne0,\\
&~1~~~~~~~~~~~&&{\rm if}~~~~~{\rm Im}\lambda^+_1{\rm Im}\lambda^-_1=0,
\end{aligned}
\right.
\label{eigen2}
\end{eqnarray}
${\rm Re}$ and ${\rm Im}$ denote the real and imaginary part of eigenvalues respectively. We have the first theorem as follows.

\begin{thm}
Every $Z\in\Omega_1$ is locally $\Sigma$-equivalent to its corresponding piecewise linear vector
field $Z_L$ of form {\rm (\ref{pwl})} near the origin. Moreover, the local phase portrait of
$Z$ near the origin is one of the 11 phase portraits presented
in Figure~\ref{localphaseportraits} in the sense of $\Sigma$-equivalence.
\label{normalform}
\end{thm}

\begin{figure}
  \begin{minipage}[t]{0.24\linewidth}
  \centering
  \includegraphics[width=1.4in]{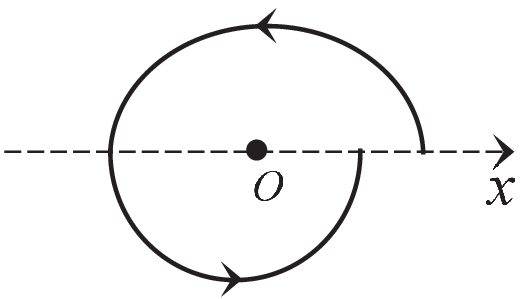}
  \caption*{{\small (FF-1)}}
  \end{minipage}
  \begin{minipage}[t]{0.24\linewidth}
  \centering
  \includegraphics[width=1.39in]{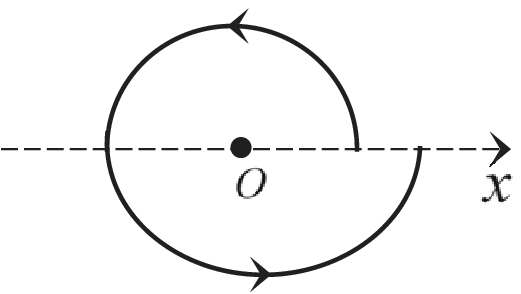}
  \caption*{{\small (FF-2)}}
  \end{minipage}
  \begin{minipage}[t]{0.24\linewidth}
  \centering
  \includegraphics[width=1.4in]{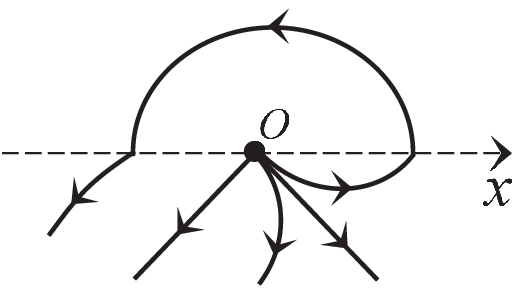}
  \caption*{{\small (FN-1)}}
  \end{minipage}
  \begin{minipage}[t]{0.24\linewidth}
  \centering
  \includegraphics[width=1.42in]{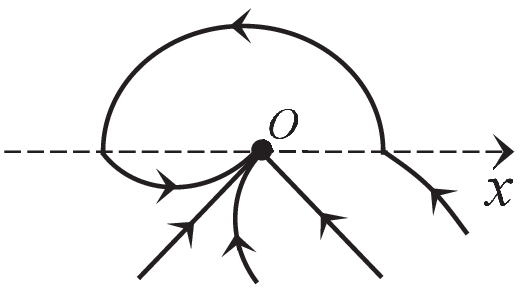}
  \caption*{{\small (FN-2)}}
  \end{minipage}
  \begin{minipage}[t]{0.24\linewidth}
  \centering
  \includegraphics[width=1.4in]{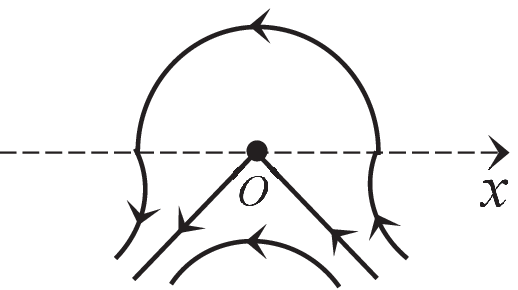}
  \caption*{{\small (FS)}}
  \end{minipage}
  \begin{minipage}[t]{0.24\linewidth}
  \centering
  \includegraphics[width=1.4in]{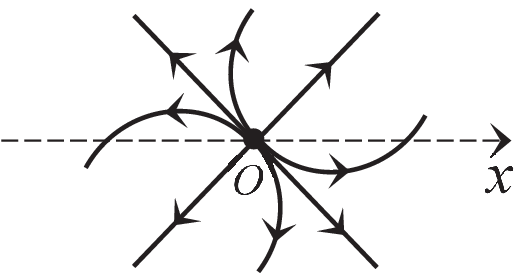}
  \caption*{{\small (NN-1)}}
  \end{minipage}
  \begin{minipage}[t]{0.24\linewidth}
  \centering
  \includegraphics[width=1.4in]{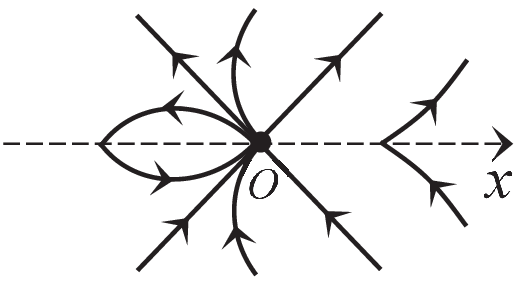}
  \caption*{{\small (NN-2)}}
  \end{minipage}
  \begin{minipage}[t]{0.24\linewidth}
  \centering
  \includegraphics[width=1.4in]{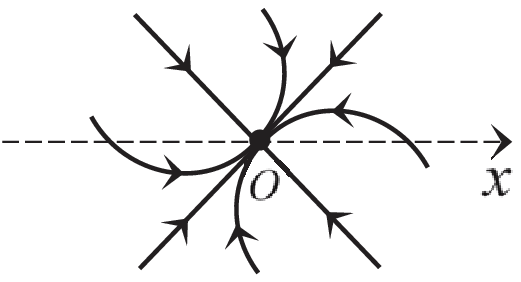}
  \caption*{{\small (NN-3)}}
  \end{minipage}
  \begin{minipage}[t]{0.24\linewidth}
  \centering
  \includegraphics[width=1.4in]{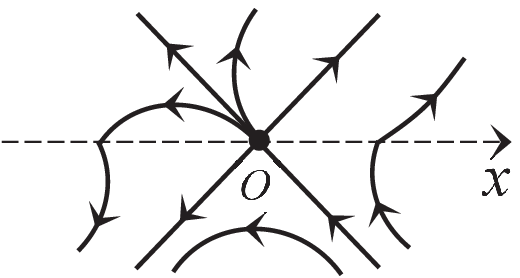}
  \caption*{{\small (NS-1)}}
  \end{minipage}~~
  \begin{minipage}[t]{0.24\linewidth}
  \centering
  \includegraphics[width=1.4in]{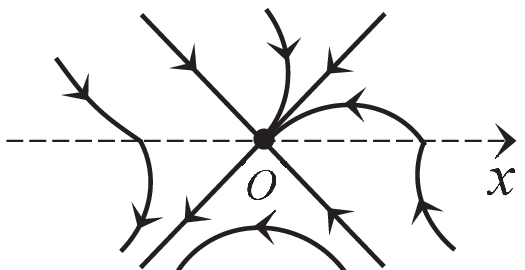}
  \caption*{{\small (NS-2)}}
  \end{minipage}~~
  \begin{minipage}[t]{0.24\linewidth}
  \centering
  \includegraphics[width=1.42in]{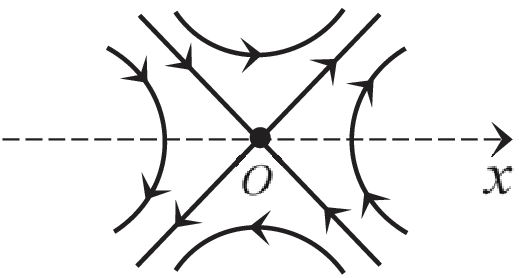}
  \caption*{{\small (SS)}}
  \end{minipage}
\caption{{\small Local phase portraits of $Z\in\Omega_1$ near the origin.}}
\label{localphaseportraits}
\end{figure}

Theorem~\ref{normalform} is proved in Section 3, where we present a normal form for each one of these $11$ kinds of
phase portraits shown in Figure~\ref{localphaseportraits}.
We remark that the first part of Theorem~\ref{normalform} can be regarded as a
generalisation of \cite[Theorem 4.7]{ZZF} from smooth vector fields to piecewise smooth vector fields. We clarify some differences between the requirements for eigenvalues in these two theorems as follows.
In \cite[Theorem 4.7]{ZZF} it is required that all eigenvalues of the Jacobian matrix at a smooth equilibrium have nonzero real part in order that
the smooth vector field is topologically equivalent to its linear part near this equilibrium.
However, in Theorem~\ref{normalform} we require that all eigenvalues of the Jacobian matrixes $A^+$ and $A^-$ at $O$, namely the non-smooth equilibrium, satisfy
$$\lambda^\pm_1\lambda^\pm_2\ne0,~~~~~\lambda^+_1\ne\lambda^+_2,~~~~~~\lambda^-_1\ne\lambda^-_2,~~~~~\ell\ne0$$
by the definition of $\Omega_1$ given in (\ref{subsetdefinition}).
Comparing the requirements of \cite[Theorem 4.7]{ZZF} with our Theorem~\ref{normalform}, we see that
\cite[Theorem 4.7]{ZZF} does not allow pure imaginary eigenvalues but Theorem~\ref{normalform} allows.
On the other hand, by \cite[Theorem B]{CGP} or \cite[Theorem 1.2]{HZ} the condition $\ell\ne0$ in Theorem~\ref{normalform}
excludes the case that $O$ is a non-smooth center of the linear part.
It is not hard to give an example showing the non-equivalence when $O$ is a non-smooth center of the linear part.
Another difference is that \cite[Theorem 4.7]{ZZF} allows the Jacobian matrix to have the same eigenvalue,
but Theorem~\ref{normalform} does not allow this for both Jacobian matrices $A^+$ and $A^-$. We give an example to show the non-equivalence when the Jacobian matrix $A^+$ or $A^-$ has the same eigenvalue in Section 3.

Our second goal is to study the structural stability of $Z\in \Omega_0$ in the sense of $\Sigma$-equivalence, i.e.,
{\it $\Sigma$-structural stability} as defined in \cite[p.1978]{MG1}. Usually, $Z\in\Omega_0$ is not $\Sigma$-structurally stable when the perturbation is inside $\Omega$ because $O$ can be destroyed under such a perturbation and the so-called boundary equilibrium bifurcation occurs \cite{YAK}. Thus the only interest is to consider the $\Sigma$-structural stability of $Z\in\Omega_0$ with respect to $\Omega_0$, i.e., the perturbation is inside $\Omega_0$. In particular, we focus on the local $\Sigma$-structural stability of $Z\in\Omega_0$ near $O$. Roughly speaking, $Z\in\Omega_0$ is said to be {\it locally $\Sigma$-structurally stable with respect to $\Omega_0$ near $O$}
if any vector field that lies in a sufficiently small neighborhood of $Z$ contained in $\Omega_0$ is locally $\Sigma$-equivalent to $Z$ near $O$.

\begin{thm}
$Z\in\Omega_0$ is locally $\Sigma$-structurally stable with respect to $\Omega_0$ near the origin if and only if $Z\in\Omega_1$,
where $\Omega_0$ is defined above {\rm (\ref{adc})} and $\Omega_1$ is defined in {\rm(\ref{subsetdefinition})}.
\label{stability}
\end{thm}
\vspace{-13pt}
Theorem~\ref{stability} is proved in Section 4.

The third goal of this paper is devoted to the study of limit cycle bifurcations, more precisely, identify the existence and number of crossing limit cycles bifurcating from the non-smooth equilibrium $O$ of a piecewise smooth vector field $Z=(X, Y)\in\Omega_0$. Here a limit cycle is said to be a {\it crossing limit cycle} if it intersects the switching line $\Sigma$ only at crossing points (see Section 2).
Many works about limit cycle bifurcations are done for the case that $O$ is of focus-focus type, i.e., an equilibrium of focus type for both $X$ and $Y$.
See, e.g., \cite{CGP, ZK, KM, Xingwu, XCZ, CNTT, LHHH} for the perturbations in $\Omega_0$ and \cite{HZ, YMH} for the perturbations in $\Omega$.
Such bifurcation is analogous to the Hopf bifurcation of smooth vector fields. Then a natural question is whether limit cycles can bifurcate from $O$ for other cases, for instance $O$ is of focus-saddle type, focus-node type, etc.
Since bifurcations usually depend on the type of local phase portraits of the unperturbed systems and
there exist many kinds of possibilities as obtained in Theorem~\ref{normalform},
in this paper we do not establish the bifurcation diagrams one by one but give some universal results on the limit cycle bifurcations for all unperturbed vector fields in $\Omega_0$.

\begin{thm}
For $\Omega_0$ defined above {\rm(\ref{adc})} and its subset $\Omega_1$ defined in {\rm(\ref{subsetdefinition})}, the following statements hold.
\vspace{-13pt}
\begin{description}
\setlength{\itemsep}{-0.8mm}
\item{\rm(1)} For any $Z\in\Omega_0$ and any small neighborhood $\mathcal{N}\subset\Omega$ of $Z$, there exists a vector field in $\mathcal{N}$ having a crossing limit cycle bifurcating from the non-smooth equilibrium $O$ of $Z$.
\item{\rm(2)} There exists a $Z_0\in\Omega_1$ {\rm(}resp. $\Omega_0\setminus\Omega_1${\rm)} such that, for any $m\in\mathbb{N}^+\cup\{\infty\}$ and any small neighborhood $\mathcal{N}\subset\Omega$ of $Z_0$, there exists a vector field in $\mathcal{N}$ having exactly $m$ hyperbolic crossing limit cycles bifurcating from the non-smooth equilibrium $O$ of $Z_0$.
    \end{description}
\label{bifurcation}
\end{thm}
\vspace{-13pt}

Theorem~\ref{bifurcation} is proved in Section 5.
Note that even though our main motivation is to consider the case of piecewise smooth vector fields, the set $\Omega_0$ also includes the smooth
vector fields with $X\equiv Y$ having $O$ as a non-degenerate equilibrium. Thus it follows from the statement (1) of Theorem~\ref{bifurcation}
that limit cycles can bifurcate from a rough focus, saddle or node of smooth vector fields under non-smooth perturbations. This is impossible under smooth perturbations.

This paper is organized as follows. In Section 2 we shortly recall basic notions and results on piecewise smooth vector fields.
In Section 3 we give the proof of Theorem~\ref{normalform}, and an example showing that the vector field in $\Omega_0$ might not
be locally $\Sigma$-equivalent to its linear part near the origin if the Jacobian matrix $A^+$ or $A^-$ has the same eigenvalue.
The proofs of Theorems~\ref{stability} and \ref{bifurcation} are given in Sections 4 and 5, respectively.

\section{Preliminaries}
\setcounter{equation}{0}
\setcounter{lm}{0}
\setcounter{thm}{0}
\setcounter{rmk}{0}
\setcounter{df}{0}
\setcounter{cor}{0}

For the sake of completeness, in this section we shortly review some basic notions and results on piecewise smooth vector fields,
especially Filippov vector fields. Section 2.1 contains the definitions of vector field $Z_\Sigma$ on $\Sigma$ and
all kinds of singularities. Moreover, the local $\Sigma$-equivalence is also clarified in Section 2.1. In Section 2.2 we state
the pseudo-Hopf bifurcation for a special class of piecewise smooth vector fields in order to prove our results conveniently.

\subsection{Notions on piecewise smooth vector fields}
Consider the piecewise smooth vector field $Z\in\Omega$ given in (\ref{sysp}). First we clarify the definition of vector field $Z_\Sigma$ on $\Sigma$ by the Filippov convention \cite{AFF}. To do this, $\Sigma$ is divided into
the {\it crossing set}
$$\Sigma^c=\{(x, y)\in\Sigma: X_2(x, y)\cdot Y_2(x, y)>0\},$$
and the {\it sliding set}
$$\Sigma^s=\{(x, y)\in\Sigma: X_2(x, y)\cdot Y_2(x, y)\le0\},$$
as in \cite{YAK, AFF}.
The points in $\Sigma^c$ and $\Sigma^s$ are called {\it crossing points} and {\it sliding points} respectively.
For $(x, y)\in\Sigma^c$, $X$ and $Y$ are both transversal to $\Sigma$ and their normal components have the same sign, so that the orbit passing through $(x, y)$ crosses $\Sigma$ at $(x, y)$ and it is a continuous, but non-smooth curve. This means that we can define $Z_\Sigma$ at $(x, y)$ as any one of $X$ and $Y$. For concreteness, in this paper we specify
$$
Z_\Sigma(x, y)=\left\{
\begin{aligned}
&Y(x, y)~~~&&{\rm if} ~~(x, y)\in\Sigma^c,~~ X_2(x, y)<0,\\
&X(x, y)~~~&&{\rm if} ~~(x, y)\in\Sigma^c,~~ X_2(x, y)>0.\\
\end{aligned}
\right.
$$
For $(x, y)\in\Sigma^s$, either the normal components of $X$ and $Y$ to $\Sigma$ have the opposite sign or at least one of them vanishes. In this case $Z_\Sigma$ is defined such that it is tangent to $\Sigma^s$. Particularly, if $Y_2(x, y)\ne X_2(x, y)$,
$$Z_\Sigma(x, y)=\left(\frac{Y_2(x, y)X_1(x, y)-X_2(x, y)Y_1(x, y)}{Y_2(x, y)-X_2(x, y)},~0\right)$$
by \cite{AFF, YAK}, while if $Y_2(x, y)=X_2(x, y)=0$, namely $(x, y)$ is a {\it singular sliding point} (see \cite{YAK}),
we always assume $Z_\Sigma(x, y)=(0, 0)$
in this paper. Sometimes, $Z_\Sigma$ restricted on $\Sigma^s$, denoted by $Z^s$, is called the {\it sliding vector field} of $Z$ and the corresponding equilibria are said to be {\it pseudoequilibria}. Having the definition of $Z_\Sigma$, the flow of $Z$ can be obtained by concatenating the flows of $X, Y$ and $Z_\Sigma$ as stated in \cite{YAK}.

In the switching line $\Sigma$, the boundary $\partial\Sigma^s$ of $\Sigma^s$ plays an important role in the dynamical analysis of piecewise smooth vector fields. Let $q\in\partial\Sigma^s$. If $X_2(q)=0, X(q)\ne0$ (resp. $Y_2(q)=0, Y(q)\ne0$), then $q$ is called a {\it tangency point} of $X$ (resp. $Y$), see \cite{YAK}. In addition, a tangency point $q$ of $X$ is called a {\it fold point} if $X_1(q)X_{2x}(q)\ne0$ and it is said to be {\it visible} (resp. {\it invisible}) when $X_1(q)X_{2x}(q)>0$ (resp. $X_1(q)X_{2x}(q)<0$). The above notions can be similarly defined for $Y$.
If $q$ is a fold point of both $X$ and $Y$, we call it a {\it fold-fold point} of $Z$, which can be divided into visible-visible, invisible-invisible and visible-invisible types. If $X(q)=0$ (resp. $Y(q)=0$), $q$ is called a {\it boundary equilibrium} of $X$ (resp. $Y$). Clearly,
a boundary equilibrium must be a pseudoequilibrium.

Regarding piecewise smooth vector fields, there are two types of equivalences, i.e., topological equivalence and $\Sigma$-equivalence. We adopt the latter in this paper as it was indicated in Section 1, see \cite[Definition 2.20]{MG1} and \cite[Definition 2.30]{MD} for the definition of $\Sigma$-equivalence.
Since we deal with the local dynamics of $Z\in\Omega_0$ near the origin, namely the non-smooth equilibrium, we can localize the definition of the $\Sigma$-equivalence as follows.
\begin{df}
Consider two piecewise smooth vector fields $Z_1$ and $Z_2$ in $\Omega_0$. We say that $Z_1$ and $Z_2$ are {\rm locally $\Sigma$-equivalent} near the origin if
\vspace{-13pt}
\begin{description}
\setlength{\itemsep}{-0.8mm}
\item{\rm(1)} $Z_1$ and $Z_2$ are locally topologically equivalent near the origin, i.e., there exist two neighborhoods $U$ and $V$ of the origin, and a homeomorphism $H: U\rightarrow V$ such that $H$ maps the orbits of $Z_1$ in $U$ onto the orbits of $Z_2$ in $V$, preserving the direction of time; and
\item{\rm(2)} the homeomorphism $H$ sends $\Sigma\cap U$ to $\Sigma\cap V$.
\end{description}
\label{signaequil}
\end{df}
\vspace{-13pt}
As a result, the definition of local $\Sigma$-equivalence gives rise to the definition of {\it local $\Sigma$-structural stability} of $Z\in\Omega_0$
with respect to $\Omega_0$ near the origin, that is, $Z\in\Omega_0$ is said to be locally $\Sigma$-structurally stable with respect to $\Omega_0$ near the origin, if any vector field that lies in a sufficiently small neighborhood of $Z$ contained in $\Omega_0$ is locally $\Sigma$-equivalent to $Z$ near the origin.

\subsection{Pseudo-Hopf bifurcation}

It is well known that the Hopf bifurcation of smooth vector fields is a main tool to produce limit cycles, where limit cycles bifurcate from a weak focus as the stability of this focus changes. In piecewise smooth vector fields there exists a similar phenomenon, called {\it pseudo-Hopf bifurcation} (see, e.g., \cite{HZ, RLS, MG1, CNTT, JCJLV}), where limit cycles are created from a {\it pseudo-focus} as the stability of a sliding segment changes, see Figure~\ref{pseudohopfbifurcation}. Here a point in the switching line is said to be a stable (resp. unstable) pseudo-focus if all orbits near this point turn around and tend to it as the time increases (resp. decreases) as defined in \cite{CGP}.
In order to prove the results of this paper conveniently, we adopt the version given in \cite[Proposition 2.3]{CNTT} by considering
the special one-parametric piecewise smooth vector field
\begin{eqnarray}
Z_\delta(x, y)=\left\{
\begin{aligned}
&X(x, y)~~~~~&&{\rm if}~~y>0,\\
&Y(x, y)+(0, \delta)^\top~~~~~~~&&{\rm if}~~y<0,
\end{aligned}
\right.
\label{ejeeewer}
\end{eqnarray}
where $X=(X_1, X_2)$ and $Y=(Y_1, Y_2)$ are $\mathcal{C}^1$ vector fields defined on $\mathbb{R}^2$, $\delta\in\mathbb{R}$ is a parameter.
\begin{figure}[htp]
  \begin{minipage}[t]{0.33\linewidth}
  \centering
  \includegraphics[width=1.7in]{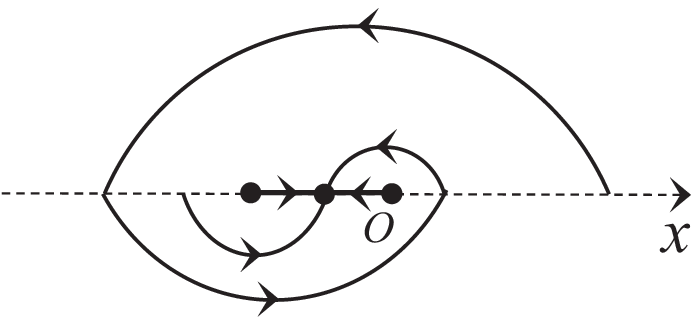}
  \caption*{{\small $\delta>0$}}
  \end{minipage}
  \begin{minipage}[t]{0.33\linewidth}
  \centering
  \includegraphics[width=1.7in]{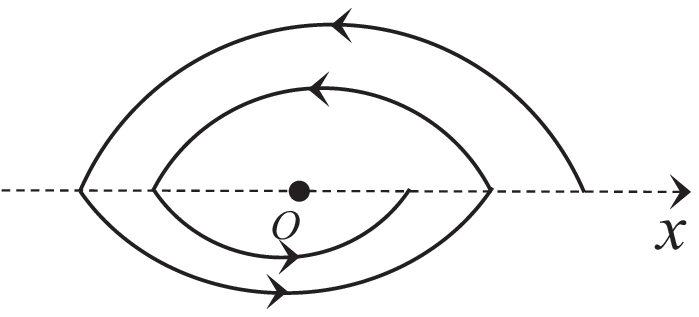}
  \caption*{{\small $\delta=0$}}
  \end{minipage}
  \begin{minipage}[t]{0.33\linewidth}
  \centering
  \includegraphics[width=1.7in]{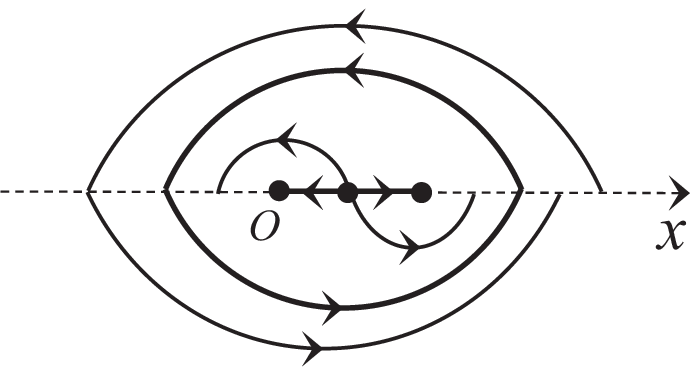}
  \caption*{{\small $\delta<0$}}
  \end{minipage}
\caption{{\small The pseudo-Hopf bifurcation of (\ref{ejeeewer}) satisfying $X_1(0, 0)<0<Y_1(0, 0)$ and the origin is stable.}}
\label{pseudohopfbifurcation}
\end{figure}

\begin{prop}
For $\delta=0$ we assume that the origin is a stable {\rm(}resp. unstable{\rm)} pseudo-focus formed by an invisible-invisible fold-fold point of the piecewise smooth vector field $Z_\delta$ and $X_1(0, 0)<0<Y_1(0, 0)$. Then the vector field $Z_\delta$ exhibits a pseudo-Hopf bifurcation at $\delta=0$ for $|\delta|$ sufficiently small, more precisely, there exists some $\delta_0>0$ such that $Z_\delta$ has a stable {\rm(}resp. unstable{\rm)} crossing limit cycle bifurcating from the origin for $-\delta_0<\delta<0$ {\rm(}resp. $\delta_0>\delta>0${\rm)} and has no crossing limit cycles for $\delta_0>\delta>0$ {\rm(}resp. $-\delta_0<\delta<0${\rm)}.
\label{pseudohopf}
\end{prop}

The proof of Proposition~\ref{pseudohopf} follows directly from the generalized Poincar\'e-Bendixson Theorem for piecewise smooth vector fields, see \cite{CTEE}.

\section{Proof of Theorem~\ref{normalform}}
\setcounter{equation}{0}
\setcounter{lm}{0}
\setcounter{thm}{0}
\setcounter{rmk}{0}
\setcounter{df}{0}
\setcounter{cor}{0}

This section is devoted to proving Theorem~\ref{normalform}. Let $Z=(X, Y)\in\Omega_0$. We start by studying the local sliding dynamics of $Z$ near the origin $O$.
\begin{lm}
For $Z=(X, Y)\in\Omega_0$ there exists a neighborhood ${\mathcal U}_0\subset {\mathcal U}$ of $O$ such that $\Sigma\cap {\mathcal U}_0$ is separated into two crossing sets by $O$. In addition, if $X_{2x}(0, 0)>0$ and $Y_{2x}(0, 0)>0$, the direction of $X$ and $Y$ on the right {\rm(}resp. left{\rm)} crossing set is upward {\rm(}resp. downward{\rm)}, while if $X_{2x}(0, 0)<0$ and $Y_{2x}(0, 0)<0$, the direction of $X$ and $Y$ on the right {\rm(}resp. left{\rm)} crossing set is downward {\rm(}resp. upward{\rm)}.
\label{slidy}
\end{lm}

\begin{proof}
Writing $X_2(x, 0)$ and $Y_2(x, 0)$ around $x=0$ as
\begin{eqnarray}
X_2(x, 0)=X_{2x}(0, 0)x+\mathcal{O}(x^2),~~~~~~~~Y_2(x, 0)=Y_{2x}(0, 0)x+\mathcal{O}(x^2),
\label{snajfaf}
\end{eqnarray}
we get $X_2(x, 0)Y_2(x, 0)=X_{2x}(0, 0)Y_{2x}(0, 0)x^2+\mathcal{O}(x^3)$.
By the definition of $\Omega_0$, we get $X_{2x}(0, 0)Y_{2x}(0, 0)>0$ and then there exists a neighborhood ${\mathcal U}_0\subset {\mathcal U}$ of $O$ such that $X_2(x, 0)Y_2(x, 0)=0$ for $(x, 0)=O$ and $X_2(x, 0)Y_2(x, 0)>0$ for $(x, 0)\in ({\mathcal U}_0\cap\Sigma)\setminus\{O\}$.
It follows from the definition of crossing set that
$\{(x, 0)\in {\mathcal U}_0\cap\Sigma: x<0\}$ and $\{(x, 0)\in {\mathcal U}_0\cap\Sigma: x>0\}$ are two crossing sets separated by $O$, i.e., the first part of Lemma~\ref{slidy} is proved. The second part is obtained directly from (\ref{snajfaf}).
\end{proof}

Our main idea for proving Theorem~\ref{normalform} is to provide a normal form for $Z\in\Omega_1\subset\Omega_0$ such that
both $Z$ and the corresponding piecewise linear vector field $Z_L$ are locally $\Sigma$-equivalent to
this normal form near the origin. Then $Z$ is locally $\Sigma$-equivalent to $Z_L$ near the origin, and the local phase portrait of $Z$ is the phase portrait of this normal form in the sense of $\Sigma$-equivalence. This concludes the proof of Theorem~\ref{normalform}.
Therefore, in what follows we will study the normal forms of $Z\in\Omega_1$ using the method introduced in \cite{MG1, TCJ, TCJLC}. Such a method has been successfully applied to obtain the normal forms of piecewise smooth vector fields in $\Omega$ near a codimension-zero (resp. codimension-one) singularity in \cite{MG1} (resp. \cite{TCJ, TCJLC}), and near a $\Sigma$-center in \cite{CT, LXZ}.

To this end we classify $\Omega_1$ into the following six subsets:
\vspace{-13pt}
\begin{description}
\setlength{\itemsep}{-0.8mm}
\item{}$\Omega_{ff}=\{Z\in\Omega_1: \lambda^\pm_1,\lambda^\pm_2 \in \mathbb{C}\setminus \mathbb{R}\}$,
\item{}$\Omega_{fn}=\{Z\in\Omega_1: {\rm either}~\lambda^+_1,\lambda^+_2\in \mathbb{C}\setminus \mathbb{R}, \lambda^-_1, \lambda^-_2\in \mathbb{R}, \lambda^-_1 \lambda^-_2>0~ {\rm or}~ \lambda^-_1,\lambda^-_2\in \mathbb{C}\setminus \mathbb{R}, \lambda^+_1, \lambda^+_2\in \mathbb{R}, \lambda^+_1 \lambda^+_2>0\}$,
\item{}$\Omega_{fs}=\{Z\in\Omega_1:{\rm either}~\lambda^+_1,\lambda^+_2\in \mathbb{C}\setminus \mathbb{R}, \lambda^-_1, \lambda^-_2\in \mathbb{R}, \lambda^-_1 \lambda^-_2<0~ {\rm or}~ \lambda^-_1,\lambda^-_2\in \mathbb{C}\setminus \mathbb{R}, \lambda^+_1, \lambda^+_2\in \mathbb{R}, \lambda^+_1 \lambda^+_2<0\}$,
\item{}$\Omega_{nn}=\{Z\in\Omega_1: \lambda^\pm_1,\lambda^\pm_2\in\mathbb{R}, \lambda^+_1\lambda^+_2>0, \lambda^-_1\lambda^-_2>0\}$,
\item{}$\Omega_{ns}=\{Z\in\Omega_1: \lambda^\pm_1, \lambda^\pm_2\in\mathbb{R}, {\rm either}~\lambda^+_1\lambda^+_2>0, \lambda^-_1\lambda^-_2<0 ~{\rm or}~\lambda^+_1\lambda^+_2<0, \lambda^-_1\lambda^-_2>0\}$,
\item{}$\Omega_{ss}=\{Z\in\Omega_1: \lambda^\pm_1,\lambda^\pm_2\in\mathbb{R}, \lambda^+_1\lambda^+_2<0, \lambda^-_1\lambda^-_2<0\}$.
\end{description}
\vspace{-13pt}
Clearly,
$$
\Omega_1=\Omega_{ff}\cup\Omega_{fn}\cup\Omega_{fs}\cup\Omega_{nn}\cup\Omega_{ns}\cup\Omega_{ss}.
$$
Now we study the normal forms for $Z=(X, Y)\in\Omega_{ff}, \Omega_{fn}, \Omega_{fs}, \Omega_{nn}, \Omega_{ns}$ and $\Omega_{ss}$, respectively.

\begin{lm}
If $Z=(X, Y)\in\Omega_{ff}$, then $Z$ is locally $\Sigma$-equivalent to $Z_{ff}=(X_{ff}, Y_{ff})\in\Omega_{ff}$ near the origin,
where
$$X_{ff}(x, y)=(\alpha x-y, x+\alpha y),\qquad Y_{ff}(x, y)=(\alpha x-y, x+\alpha y),$$
$\alpha={\rm sign}\ell$ and $\ell\ne0$ is defined in {\rm(\ref{eigen2})}.
\label{ff}
\end{lm}

\begin{proof}
Because $\Omega_{ff}\subset\Omega_1\subset\Omega_0$, $Z\in\Omega_{ff}$ satisfies (\ref{condi}) by the definition of $\Omega_0$.
Using the change $(x, y)\rightarrow(-x, y)$, we only need to consider the case
\begin{eqnarray}
X_{2x}(0, 0)>0, ~~~~~~~Y_{2x}(0, 0)>0.
\label{onecase}
\end{eqnarray}
Hence, $\Sigma\cap {\mathcal U}_0$ is separated into two crossing sets by $O$, and
the direction of $X$ and $Y$ on the right {\rm(}resp. left{\rm)} crossing set is upward {\rm(}resp. downward{\rm)}
as it is seen in Lemma~\ref{slidy}.
Recalling \cite[Theorem B]{CGP} and \cite[Theorem 1.2]{HZ}, we obtain that $O$ is a stable pseudo-focus if $\ell<0$ and an unstable pseudo-focus if $\ell>0$ for $Z\in\Omega_{ff}$ satisfying (\ref{onecase}), see Figure~\ref{ffn}. For $Z_{ff}\in\Omega_{ff}$ it is a linear vector field, and $O$ is a stable focus as shown in (FF-1) of Figure~\ref{localphaseportraits} if $\alpha=-1$ and an unstable focus as shown in (FF-2) of Figure~\ref{localphaseportraits} if $\alpha=1$.

\begin{figure}
  \begin{minipage}[t]{0.5\linewidth}
  \centering
  \includegraphics[width=1.7in]{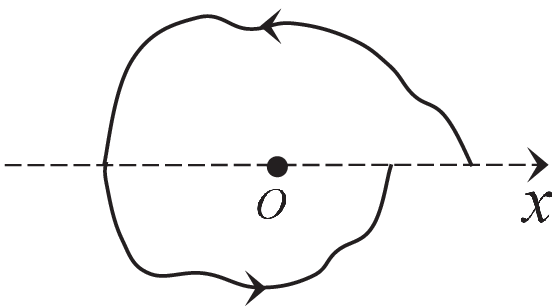}
  \caption*{(a)~ {\small $\ell<0$}}
  \end{minipage}
  \begin{minipage}[t]{0.5\linewidth}
  \centering
  \includegraphics[width=1.7in]{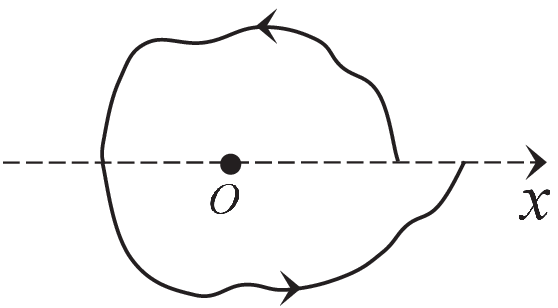}
  \caption*{(b)~ {\small $\ell>0$}}
  \end{minipage}
\caption{{\small Local phase portraits of $Z\in\Omega_{ff}$ satisfying (\ref{onecase}) near $O$.}}
\label{ffn}
\end{figure}

Next we prove this lemma for the case $\ell<0$ and $\alpha=-1$. The case $\ell>0$ and $\alpha=1$ can be treated similarly.
Consider two sufficiently small neighborhoods $U\subset {\mathcal U}_0$ and $V\subset {\mathcal U}_0$ of $O$ as shown in Figure~\ref{focusfocus}, where ${\mathcal U}_0$ is given in Lemma~\ref{slidy}, $U$ is surrounded by the closed line segment $\overline{CA}\subset\Sigma$ and the orbital arc of $Z$ from $A$ to $C$ after passing through $B$, $V$ is surrounded by the closed line segment $\overline{C_1A_1}\subset\Sigma$ and the orbital arc of $Z_{ff}$ from $A_1$ to $C_1$ after passing through $B_1$. Here overline denotes the closure. We need to construct a homeomorphism $H$ from $U$ to $V$ implying the $\Sigma$-equivalence between $Z$ with $\ell<0$ and $Z_{ff}$ with $\alpha=-1$.

\begin{figure}
  \begin{minipage}[t]{1.0\linewidth}
  \centering
  \includegraphics[width=4.8in]{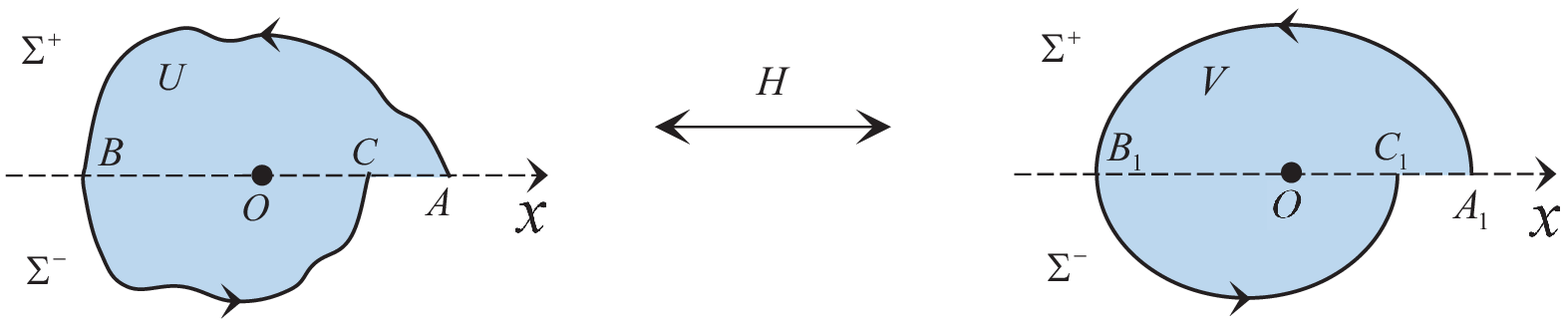}
  \end{minipage}
\caption{{\small The homeomorphism $H$ between $Z\in\Omega_{ff}$ with $\ell<0$ and $Z_{ff}$ with $\alpha=-1$.}}
\label{focusfocus}
\end{figure}

For $Z\in\Omega_{ff}$ satisfying (\ref{onecase}), $O$ is an anticlockwise rotary equilibrium of focus type of $X$ and $Y$.
Thus, given $P\in \overline{OA}$, there exist a first time $t_1=t_1(P)\ge0$ such that $\Phi^+(t_1, P)\in \overline{OB}$, and
a first time $t_2=t_2(\Phi^+(t_1, P))\ge0$ such that $\Phi^-\left(t_2, \Phi^+(t_1, P)\right)\in \overline{OC}$, where $\Phi^+$ and $\Phi^-$ denote the flows of $X$ and $Y$ respectively. This means that we can define a Poincar\'e map $\mathcal{P}: \overline{OA}\rightarrow \overline{OC}$ by
\begin{eqnarray}
\mathcal{P}(P)=\Phi^-\left(t_2, \Phi^+(t_1, P)\right).
\label{afnf}
\end{eqnarray}
In particular, $\mathcal{P}(O)=O$ and $\mathcal{P}(A)=C$, since $A$ and $C$ lie in the same orbit. Let $(x_P, 0)$ and $(\mathcal{P}_1(x_P), \mathcal{P}_2(x_P))$ be the coordinates of $P$ and $\mathcal{P}(P)$ respectively. Then $\mathcal{P}_2(x_P)=0$ and $\mathcal{P}_1(x_P)$ is given by
$$\mathcal{P}_1(x_P)=e^{\ell\pi} x_P+\mathcal{O}(x^2_P)$$
from \cite[Theorem 1.1, Theorem 1.2]{HZ}.

Similarly, denoting the flows of $X_{ff}$ and $Y_{ff}$ by $\Psi^+$ and $\Psi^-$ respectively,
we can define a Poincar\'e map $\mathcal{Q}: \overline{OA_1}\rightarrow \overline{OC_1}$ by
\begin{eqnarray}
\mathcal{Q}(P)=\Psi^-\left(s_2, \Psi^+(s_1, P)\right),
\label{dasjfn}
\end{eqnarray}
which satisfies $\mathcal{Q}(O)=O$ and $\mathcal{Q}(A_1)=C_1$, where $s_1=s_1(P)\ge0$ is the first time such that $\Psi^+(s_1, P)\in \overline{OB_1}$, and $s_2=s_2(\Psi^+(s_1, P))\ge0$ is the first time such that $\Psi^-\left(s_2, \Psi^+(s_1, P)\right)\in\overline{OC_1}$. Let $(\mathcal{Q}_1(x_P), \mathcal{Q}_2(x_P))$ be the coordinates of $\mathcal{Q}(P)$. Then $\mathcal{Q}_2(x_P)=0$ and a straightway calculation yields
$$\mathcal{Q}_1(x_P)=e^{-2\pi}x_P.$$

Since we are considering the case of $\ell<0$, according to the linearization and conjugacy theory of smooth map \cite{PH1}, $U$ and $V$ can be chosen to ensure that there exists a homeomorphism $h: [0, x_A]\rightarrow [0, x_{A_1}]$ satisfying
\begin{eqnarray}
h(0)=0, \qquad h(x_A)=x_{A_1}, \qquad h(\mathcal{P}_1(x_P))=\mathcal{Q}_1(h(x_P)),
\label{jaffse}
\end{eqnarray}
where $x_A$ and $x_{A_1}$ are the first coordinates of $A$ and $A_1$ respectively.
Consequently, we define a homeomorphism $H_0: \overline{OA}\rightarrow \overline{OA_1}$ by
\begin{eqnarray}
H_0(P)=H_0(x_P, 0)=(h(x_P), 0) \qquad {\rm for} \quad P\in \overline{OA}.
\label{ajfff}
\end{eqnarray}
Clearly, it follows from (\ref{jaffse}) that $H_0(O)=O$, $H_0(A)=A_1$ and $H_0(C)=C_1$.

Given $P\in\overline{OB}$, there exists a first time $t_3=t_3(P)\le0$ such that $\Phi^+(t_3, P)\in\overline{OA}$,
since $O$ is an anticlockwise rotary equilibrium of focus type of $X$. Then $H_0(\Phi^+(t_3, P))\in\overline{OA_1}$ and there exists a first time $s_3=s_3(H_0(\Phi^+(t_3, P)))\ge0$ such that $\Psi^+(s_3, H_0(\Phi^+(t_3, P)))\in\overline{OB_1}$ because $O$ is an anticlockwise rotary focus of $X_{ff}$. By the arc length parametrization we can identify the orbital arc of $X$ from $\Phi^+(t_3, P)$ to $P$ with the one of $X_{ff}$ from $H_0(\Phi^+(t_3, P))$ to $\Psi^+(s_3, H_0(\Phi^+(t_3, P)))$. Therefore, in this way we can define a homeomorphism $H^+: \overline{\Sigma^+\cap U}\rightarrow\overline{\Sigma^+\cap V}$ that maps $\overline{BA}$ onto $\overline{B_1A_1}$, maps the orbits of $X$ in $\overline{\Sigma^+\cap U}$ onto the orbits of $X_{ff}$ in $\overline{\Sigma^+\cap V}$ and satisfies
\begin{eqnarray}
\left.H^+\right|_{\overline{OA}}=H_0.
\label{dnjvdhf}
\end{eqnarray}

Given $P\in\overline{OC}$, there exists a first time $t_4=t_4(P)\le0$ such that $\Phi^-(t_4, P)\in\overline{OB}$. Then $H^+(\Phi^-(t_4, P))\in\overline{OB_1}$ from the definition of $H^+$, and there exists a first time $s_4=s_4(H^+(\Phi^-(t_4, P)))\ge0$ such that $\Psi^-\left(s_4, H^+(\Phi^-(t_4, P))\right)\in\overline{OC_1}$. Similarly we can identify the orbital arc of $Y$ from $\Phi^-(t_4, P)$ to $P$ with the one of $Y_{ff}$ from $H^+(\Phi^-(t_4, P))$ to $\Psi^-\left(s_4, H^+(\Phi^-(t_4, P))\right)$, and thus define a homeomorphism $H^-: \overline{\Sigma^-\cap U}\rightarrow\overline{\Sigma^-\cap V}$ that maps $\overline{BC}$ onto $\overline{B_1C_1}$, maps the orbits of $Y$ in $\overline{\Sigma^-\cap U}$ onto the orbits of $Y_{ff}$ in $\overline{\Sigma^-\cap V}$ and satisfies
\begin{eqnarray}
\left.H^-\right|_{\overline{OB}}=\left.H^+\right|_{\overline{OB}}.
\label{dnjvdhwewf}
\end{eqnarray}
Moreover, for any $P\in\overline{OC}$ we have
$$
\begin{aligned}
H^-(P)&=\Psi^-\left(s_4, H^+(\Phi^-(t_4, P))\right)=\Psi^-\left(s_4, \Psi^+(s_3, H_0(\Phi^+(t_3, \Phi^-(t_4, P))))\right)\\
&=\mathcal{Q}(H_0(\Phi^+(t_3, \Phi^-(t_4, P))))=H_0(\mathcal{P}(\Phi^+(t_3, \Phi^-(t_4, P))))\\
&=H_0(P)
\end{aligned}
$$
by (\ref{afnf}), (\ref{dasjfn}), (\ref{jaffse}), (\ref{ajfff}) and the constructions of $H^\pm$. This implies that
\begin{eqnarray}
\left.H^-\right|_{\overline{OC}}=\left.H_0\right|_{\overline{OC}}.
\label{dnejnfew}
\end{eqnarray}

Let
\begin{eqnarray}
H(P)=\left\{
\begin{aligned}
&H^+(P) \qquad &&{\rm for}\quad P\in(\Sigma^+\cup\Sigma)\cap U,\\
&H^-(P) \qquad &&{\rm for}\quad P\in(\Sigma^-\cup\Sigma)\cap U.
\end{aligned}
\right.
\label{asjfnjfec}
\end{eqnarray}
Then $H$ is a homeomorphism from $U$ to $V$ because $H^\pm$ are homeomorphisms in their domains and $\left.H^+\right|_{\overline{BC}}=\left.H^-\right|_{\overline{BC}}$ by (\ref{dnjvdhf}), (\ref{dnjvdhwewf}) and (\ref{dnejnfew}). Furthermore, the construction of $H$ ensures that $H$ maps the orbits of $Z\in\Omega_{ff}$ with $\ell<0$ in $U$ onto the orbits of $Z_{ff}$ with $\alpha=-1$ in $V$, preserving the direction of time and the switching line $\Sigma$. We eventually conclude that $Z\in\Omega_{ff}$ with $\ell<0$ and $Z_{ff}$ with $\alpha=-1$ are locally $\Sigma$-equivalent near $O$.
\end{proof}

\begin{lm}
If $Z=(X, Y)\in\Omega_{fn}$, then $Z$ is locally $\Sigma$-equivalent to $Z_{fn}=(X_{fn}, Y_{fn})\in\Omega_{fn}$ near the origin, where
$$X_{fn}(x, y)=(-y, x), \qquad Y_{fn}(x, y)=(2\beta x+y, x+2\beta y)$$
and
$$\beta=
\left\{
\begin{aligned}
&{\rm sign}(\lambda^-_1+\lambda^-_2) \qquad {\rm when}~ \lambda^-_1,\lambda^-_2\in\mathbb{R},\\
&{\rm sign}(\lambda^+_1+\lambda^+_2) \qquad {\rm when}~ \lambda^+_1,\lambda^+_2\in\mathbb{R}.
\end{aligned}
\right.
$$
\label{fn}
\end{lm}

\begin{proof}
By $(x, y)\rightarrow(x, -y)$ and $(x, y)\rightarrow(-x, y)$
we only need to consider $Z\in\Omega_{fn}$ satisfying (\ref{onecase}) and
$$\lambda^+_1,\lambda^+_2\in\mathbb{C}\setminus\mathbb{R}, \qquad \lambda^-_1,\lambda^-_2\in\mathbb{R}, \qquad\lambda^-_1\lambda^-_2>0.$$
In this case, $O$ is an equilibrium of focus type of $X$ and a node of $Y$ by \cite[Theorems 4.2, 4.3, 5.1]{ZZF}.
Thus, recalling the dynamics on $\Sigma$ given in Lemma~\ref{slidy}, we get two different types of the local phase portraits of $Z$ near $O$ depending on the sign of $\lambda^-_1+\lambda^-_2$, namely the stability of $O$ when it is regarded as an equilibrium of $Y$, see Figure~\ref{fnn}.
In Figure~\ref{fnn}(a), the strong unstable manifold $m^u_s$ lies in the left side of the weak unstable manifold $m^u_w$, while in Figure~\ref{fnn}(b),
the strong stable manifold $m^s_s$ lies in the right side of the weak stable manifold $m^s_w$. Here we use the assumption of $\lambda^-_1\ne\lambda^-_2$ for all vector fields in $\Omega_1$.
Regarding the vector field $Z_{fn}$, we easily verify that its phase portrait is the one either as shown in (FN-1) of Figure~\ref{localphaseportraits} if $\beta=1$, or as shown in (FN-2) of Figure~\ref{localphaseportraits} if $\beta=-1$.
\begin{figure}
  \begin{minipage}[t]{0.5\linewidth}
  \centering
  \includegraphics[width=1.7in]{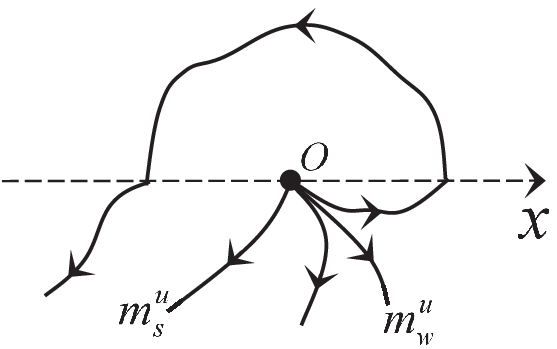}
  \caption*{(a)~ {\small $\lambda^-_1+\lambda^-_2>0$}}
  \end{minipage}
  \begin{minipage}[t]{0.5\linewidth}
  \centering
  \includegraphics[width=1.7in]{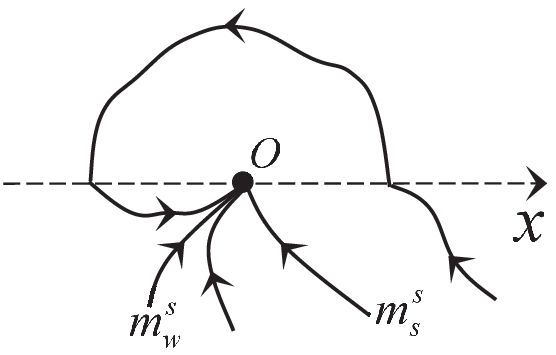}
  \caption*{(b)~ {\small $\lambda^-_1+\lambda^-_2<0$}}
  \end{minipage}
\caption{{\small Local phase portraits of $Z\in\Omega_{fn}$ satisfying (\ref{onecase}) and $\lambda^+_1,\lambda^+_2\in\mathbb{C}\setminus\mathbb{R}, \lambda^-_1,\lambda^-_2\in\mathbb{R}, \lambda^-_1\lambda^-_2>0$ near $O$.}}
\label{fnn}
\end{figure}

We only consider $\lambda^-_1+\lambda^-_2>0$ and $\beta=1$ because the case of $\lambda^-_1+\lambda^-_2<0$ and $\beta=-1$ is similar. Consider two sufficiently small neighborhoods $U\subset {\mathcal U}_0$ and $V\subset {\mathcal U}_0$ of $O$ as shown in Figure~\ref{focusnode}, where ${\mathcal U}_0$ is given in Lemma~\ref{slidy}, $U$ is surrounded by orbital arc $\widehat{AB}$ of $X$ from $A$ to $B$, and arc $\widehat{BA}$ on which $Y$ is transverse to it, $V$ is surrounded by orbital arc $\widehat{A_2B_2}$ of $X_{fn}$ from $A_2$ to $B_2$, and arc $\widehat{B_2A_2}$ on which the vector field $Y_{fn}$ is transverse to it. We need to construct a homeomorphism $H$ from $U$ to $V$ providing the $\Sigma$-equivalence between $Z\in\Omega_{fn}$ with $\lambda^-_1+\lambda^-_2>0$ and $Z_{fn}$ with $\beta=1$.
\begin{figure}
  \begin{minipage}[t]{1.0\linewidth}
  \centering
  \includegraphics[width=5in]{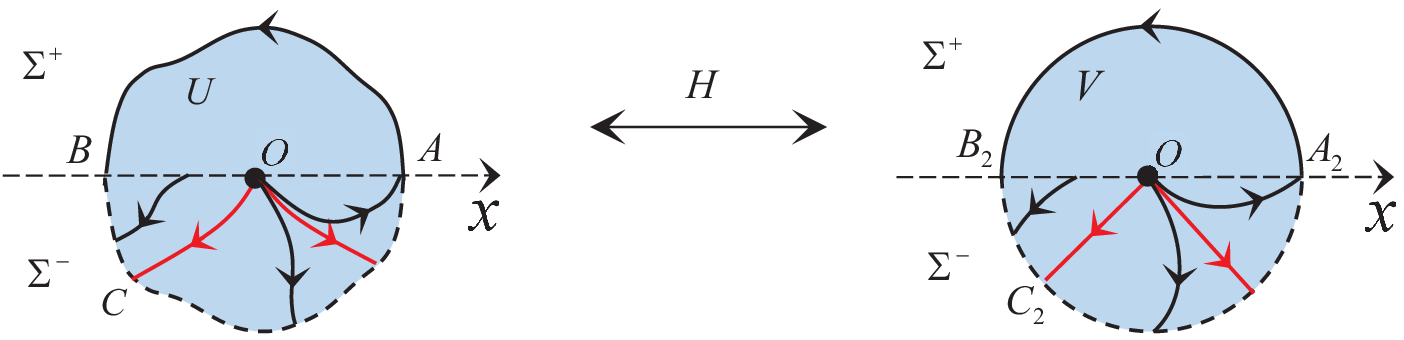}
  \end{minipage}
\caption{{\small The homeomorphism $H$ between $Z\in\Omega_{fn}$ with $\lambda^-_1+\lambda^-_2>0$ and $Z_{fn}$ with $\beta=1$.}}
\label{focusnode}
\end{figure}

By the arc length parametrization there exists a homeomorphism $H_0: \overline{OA}\rightarrow \overline{OA_2}$
such that $H_0(O)=O$ and $H_0(A)=A_2$. Since $O$ is an anticlockwise rotary equilibrium of focus type of $X$, the forward orbit of $X$ starting from $P\in\overline{OA}$ evolves in $\overline{\Sigma^+\cap U}$ until it reaches $\overline{OB}$ at a point $Q$. Then $H_0(P)\in\overline{OA_2}$. Since $O$ is an anticlockwise rotary center of $X_{fn}$, the forward orbit of $X_{fn}$ starting from $H_0(P)$ evolves in $\overline{\Sigma^+\cap V}$ until it reaches $\overline{OB_2}$ at a point $Q_2$. By the arc length parametrization we can identify the orbital arc of $X$ from $P$ to $Q$ with the one of $X_{fn}$ from $H_0(P)$ to $Q_2$. In this way we can define a homeomorphism $H_f: \overline{\Sigma^+\cap U}\rightarrow\overline{\Sigma^+\cap V}$ that maps $\overline{BA}$ onto $\overline{B_2A_2}$, maps the orbits of $X$ in $\overline{\Sigma^+\cap U}$ onto the orbits of $X_{fn}$ in $\overline{\Sigma^+\cap V}$ and satisfies
\begin{eqnarray}
\left.H_f\right|_{\overline{OA}}=H_0.
\label{itoruo}
\end{eqnarray}

Consider the region $R_{BOC}$ surrounded by $\overline{OB}$, $\widehat{BC}$ and the strong unstable manifold $\widehat{OC}$, and the corresponding region $R_{B_2OC_2}$ surrounded by $\overline{OB_2}$, $\widehat{B_2C_2}$ and the strong unstable manifold $\widehat{OC_2}$.
Given $P\in\overline{OB}$, there exists a unique point $Q\in\widehat{BC}$ such that the backward orbit of $Y$ starting from $Q$ evolves in $\overline{R_{BOC}}$ until it reaches or tends to $\overline{OB}$ at $P$, since $\widehat{OC}$ is the strong unstable manifold of the node $O$ for $Y$ and we are assuming that the vector field $Y$ on $\widehat{BA}$ is transverse to $\widehat{BA}$.
Analogously, there exists a unique point $Q_2\in\widehat{B_2C_2}$ such that the backward orbit of $Y_{fn}$ starting from $Q_2$ evolves in $\overline{R_{B_2OC_2}}$ until it reaches or tends to $\overline{OB_2}$ at $H_f(P)$. Therefore,
by the arc length parametrization again we can identify the orbital arc of $Y$ from $P$ to $Q$ with the one of $Y_{fn}$ from $H_f(P)$ to $Q_2$, and then  define a homeomorphism $H^1_n: \overline{R_{BOC}}\rightarrow \overline{R_{B_2OC_2}}$ that maps the orbits of $Y$ in $\overline{R_{BOC}}$ onto
the orbits of $Y_{fn}$ in $\overline{R_{B_2OC_2}}$ and satisfies
\begin{eqnarray}
\left.H^1_n\right|_{\overline{OB}}=\left.H_f\right|_{\overline{OB}}.
\label{uafgkfh}
\end{eqnarray}

Consider the region $R_{COA}$ surrounded by $\widehat{OC}$, $\widehat{CA}$ and $\overline{OA}$, and the corresponding region $R_{B_2OC_2}$ surrounded by $\widehat{OC_2}$, $\widehat{C_2A_2}$ and $\overline{OA_2}$. Regarding arcs $\widehat{CA}$ and $\widehat{C_2A_2}$, we obtain a homeomorphism $H^0_n: \widehat{CA}\rightarrow\widehat{C_2A_2}$ such that $H^0_n(C)=C_2$ and $H^0_n(A)=A_2$ by the arc length parametrization. Since
the choice of $U$ ensures that the vector filed $Y$ on $(\widehat{CA}\cup\overline{OA})\setminus O$ is transverse to $(\widehat{CA}\cup\overline{OA})\setminus O$, the backward orbit of $Y$ starting from $P\in(\widehat{CA}\cup\overline{OA})\setminus O$ evolves in $\overline{R_{COA}}$ and finally tends to $O$. Let $P_2=H_0(P)$ if $P\in\overline{OA}$ and $P_2=H^0_n(P)$ if $P\in\widehat{CA}$. Then the backward orbit of $Y_{fn}$ starting from $P_2$ evolves in $\overline{R_{C_2OA_2}}$ and tends to $O$. Identify the orbital arc of $Y$ from $P$ to $O$ with the orbital arc of $Y_{fn}$ from $P_2$ to $O$. In this way we can define a
homeomorphism $H^2_n: \overline{R_{COA}}\rightarrow \overline{R_{C_2OA_2}}$ that maps the orbits of $Y$ in $\overline{R_{COA}}$ onto the orbits of $Y_{fn}$ in $\overline{R_{C_2OA_2}}$ and satisfies
\begin{eqnarray}
\left.H^2_n\right|_{\widehat{OC}}=\left.H^1_n\right|_{\widehat{OC}}, \qquad \left.H^2_n\right|_{\overline{OA}}=H_0, \qquad \left.H^2_n\right|_{\widehat{CA}}=H^0_n.
\label{urcnkfh}
\end{eqnarray}

Joining the homeomorphisms $H^1_n$ and $H^2_n$, by (\ref{itoruo}), (\ref{uafgkfh}) and (\ref{urcnkfh}) we obtain that
$$
H_n(P)=\left\{
\begin{aligned}
&H^1_n(P) \qquad &&{\rm for} \quad P\in \overline{R_{BOC}},\\
&H^2_n(P) \qquad &&{\rm for} \quad P\in \overline{R_{COA}},
\end{aligned}
\right.
$$
is a homeomorphism from $\overline{\Sigma^-\cap U}$ to $\overline{\Sigma^-\cap V}$ that maps the orbits of $Y$ in $\overline{\Sigma^-\cup U}$ onto the orbits of $Y_{fn}$ in $\overline{\Sigma^-\cup V}$ and satisfies $\left.H_n\right|_{\overline{BA}}=\left.H_f\right|_{\overline{BA}}$.
Consequently, the homeomorphisms $H_n$ and $H_f$ form a homeomorphism $H: U\rightarrow V$ that
maps the orbits of $Z\in\Omega_{fn}$ with $\lambda^-_1+\lambda^-_2>0$ in $U$
onto the orbits of $Z_{fn}$ with $\beta=1$ in $V$, preserving the direction of time and the switching line $\Sigma$. This concludes the proof of Lemma~\ref{fn}.
\end{proof}

\begin{lm}
If $Z=(X, Y)\in\Omega_{fs}$, then $Z$ is locally $\Sigma$-equivalent to $Z_{fs}=(X_{fs}, Y_{fs})\in\Omega_{fs}$ near the origin, where
$$X_{fs}(x, y)=(-y, x), \qquad Y_{fs}(x, y)=(y, x).$$
\label{fsfsfs}
\end{lm}

\begin{proof}
Using the changes $(x, y)\rightarrow(x, -y)$ and $(x, y)\rightarrow(-x, y)$, we only need to consider $Z\in\Omega_{fs}$ satisfying (\ref{onecase}) and $$\lambda^+_1, \lambda^+_2\in\mathbb{C}\setminus\mathbb{R},\qquad \lambda^-_1, \lambda^-_2\in\mathbb{R}, \qquad \lambda^-_1\lambda^-_2<0.$$
In this case, $O$ is an equilibrium of focus type of $X$ and a saddle of $Y$ by \cite[Theorems 4.2, 4.4, 5.1]{ZZF}.
Reviewing the dynamics on $\Sigma$ given in Lemma~\ref{slidy}, we depict the local phase portrait of $Z$ near $O$ as shown in Figure~\ref{fsn}. The phase portrait of the vector field $Z_{fs}$ is as shown in (FS) of Figure~\ref{localphaseportraits}.

\begin{figure}
  \begin{minipage}[t]{1.0\linewidth}
  \centering
  \includegraphics[width=1.7in]{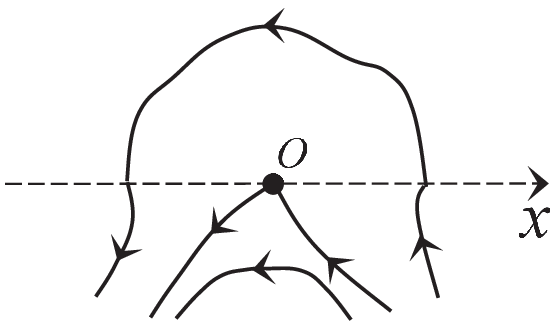}
  \end{minipage}
  \caption{{\small Local phase portrait of $Z\in\Omega_{fs}$ satisfying (\ref{onecase}) and $\lambda^+_1, \lambda^+_2\in\mathbb{C}\setminus\mathbb{R}, \lambda^-_1, \lambda^-_2\in\mathbb{R}, \lambda^-_1\lambda^-_2<0$ near $O$.}}
\label{fsn}
\end{figure}

Consider two sufficiently small neighborhoods $U\subset {\mathcal U}_0$ and $V\subset {\mathcal U}_0$ of $O$ as shown in Figure~\ref{focussaddle}, where $\widehat{AB}$ and $\widehat{A_3B_3}$ are the corresponding orbital arcs, $\widehat{BA}$ (resp. $\widehat{B_3A_3}$) is the arc where the vector field $Y$ (resp. $Y_{fs}$) is transverse to it.
\begin{figure}
  \begin{minipage}[t]{1.0\linewidth}
  \centering
  \includegraphics[width=4.8in]{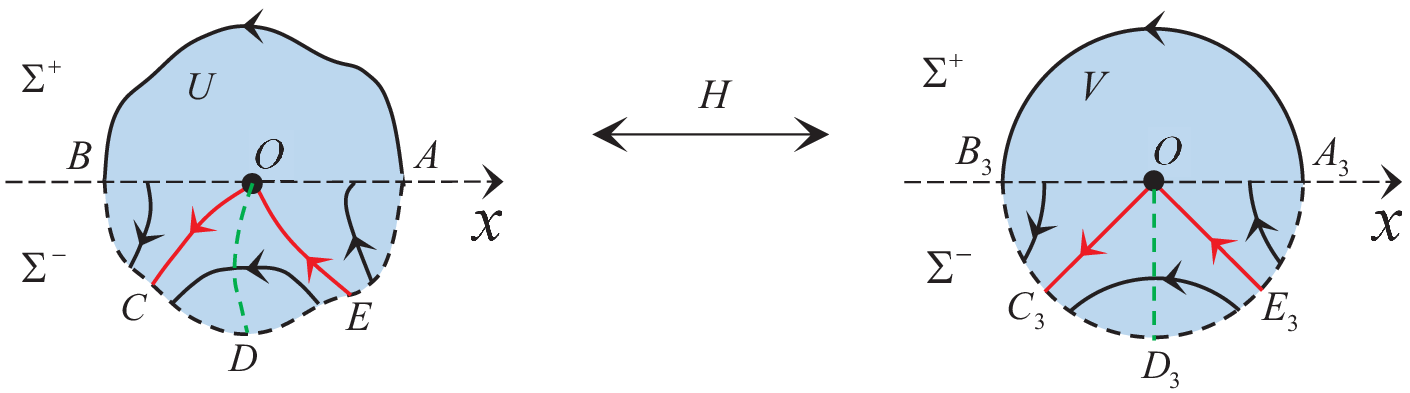}
  \end{minipage}
\caption{{\small The homeomorphism $H$ between $Z\in\Omega_{fs}$ and $Z_{fs}$.}}
\label{focussaddle}
\end{figure}
As done in the proof of Lemma~\ref{fn}, we can define a homeomorphism $H_f: \overline{\Sigma^+\cap U}\rightarrow\overline{\Sigma^+\cap V}$ that maps
$\overline{BA}$ onto $\overline{B_3A_3}$, and maps the orbits of $X$ in $\overline{\Sigma^+\cap U}$ onto the orbits of $X_{fs}$ in $\overline{\Sigma^+\cap V}$.

In order to complete this proof, next we construct a homeomorphism $H_s: \overline{\Sigma^-\cap U}\rightarrow\overline{\Sigma^-\cap V}$ that maps the orbits of $Y$ in $\overline{\Sigma^-\cap U}$ onto the orbits of $X_{fs}$ in $\overline{\Sigma^-\cap V}$ and satisfies $\left.H_s\right|_{\overline{BA}}=\left.H_f\right|_{\overline{BA}}$.
Let
$$\widehat{OD}=\{(x, y)\in\overline{\Sigma^-\cap U}: Y_2(x, y)=0\}, \qquad\overline{OD_3}=\{(x, y)\in\overline{\Sigma^-\cap V}: x=0\},$$
where $Y_2$ is the ordinate of $Y$.
Then there exists a homeomorphism $H^0_s: \widehat{OD}\rightarrow\overline{OD_3}$ such that $H^0_s(O)=O$ and $H^0_s(D)=D_3$ by the arc length parametrization.
Consider the region $R_{BOD}$ surrounded by $\overline{OB}$, $\widehat{BD}$ and $\widehat{OD}$, and the region
$R_{B_3OD_3}$ surrounded by $\overline{OB_3}$, $\widehat{B_3D_3}$ and $\overline{OD_3}$. Given $P\in\overline{OB}\cup\widehat{OD}$, there exists a unique point $Q\in\widehat{BD}$ such that the backward orbit of $Y$ starting from $Q$ evolves in $\overline{R_{BOD}}$ until it either reaches $(\overline{OB}\cup\widehat{OD})\setminus O$ when $P\ne O$ or tends to $O$ when $P=O$, since we require that the vector field $Y$ on $\widehat{BD}$ is transverse to $\widehat{BD}$. Let $P_3=H_f(P)$ if $P\in\overline{OB}$ and $P_3=H^0_s(P)$ if $P\in\widehat{OD}$. We obtain a unique point $Q_3\in\widehat{B_3D_3}$ such that the backward orbit of $Y_{fs}$ starting from $Q_3$ evolves in $\overline{R_{B_3OD_3}}$ until it reaches or tends to $P_3$. The arc length parametrization allows to identify the orbital arc of $Y$ from $Q$ to $P$ and the one of $Y_{fs}$ from $Q_3$ to $P_3$. In this way we
can define a homeomorphism $H^1_s: \overline{R_{BOD}}\rightarrow\overline{R_{B_3OD_3}}$ that maps the orbits of $Y$ in $\overline{R_{BOD}}$ onto the orbits of $Y_{fs}$ in $\overline{R_{B_3OD_3}}$ and satisfies
\begin{eqnarray}
\left.H^1_s\right|_{\overline{OB}}=\left.H_f\right|_{\overline{OB}}, \qquad \left.H^1_s\right|_{\widehat{OD}}=H^0_s.
\label{anfjfcdsfm}
\end{eqnarray}

A similar argument to the last paragraph yields a homeomorphism $H^2_s: \overline{R_{DOA}}\rightarrow\overline{R_{D_3OA_3}}$ that maps the orbits of $Y$ in $\overline{R_{DOA}}$ onto the orbits of $Y_{fs}$ in $\overline{R_{D_3OA_3}}$ and satisfies
\begin{eqnarray}
\left.H^2_s\right|_{\overline{OA}}=\left.H_f\right|_{\overline{OA}}, \qquad \left.H^2_s\right|_{\widehat{OD}}=H^0_s.
\label{anfjfcafdsfm}
\end{eqnarray}
Thus, joining the homeomorphisms $H^1_s$ and $H^2_s$ we construct $H_s$ as
$$
H_s(P)=\left\{
\begin{aligned}
&H^1_s(P) \qquad &&{\rm for} \quad P\in \overline{R_{BOD}},\\
&H^2_s(P) \qquad &&{\rm for} \quad P\in \overline{R_{DOA}}.
\end{aligned}
\right.
$$
From (\ref{anfjfcdsfm}) and (\ref{anfjfcafdsfm}), it follows that $H_s$ is a homeomorphism from $\overline{\Sigma^-\cap U}$ to $\overline{\Sigma^-\cap V}$ maps the orbits of $Y$ in $\overline{\Sigma^-\cap U}$ onto the orbits of $X_{fs}$ in $\overline{\Sigma^-\cap V}$ and satisfies $\left.H_s\right|_{\overline{BA}}=\left.H_f\right|_{\overline{BA}}$.

Consequently, the homeomorphisms $H_s$ and $H_f$ directly form a
homeomorphism $H: U\rightarrow V$ that maps the orbits of $Z\in\Omega_{fs}$ in $U$ onto the orbits of $Z_{fs}$ in $V$, preserving the direction of time and the switching line $\Sigma$. This proves Lemma~\ref{fsfsfs}.
\end{proof}

\begin{lm}
If $Z=(X, Y)\in\Omega_{nn}$, then $Z$ is locally $\Sigma$-equivalent to $Z_{nn}=(X_{nn}, Y_{nn})\in\Omega_{nn}$ near the origin, where
$$X_{nn}(x, y)=(2\gamma x+y, x+2\gamma y), \qquad Y_{nn}(x, y)=(2\eta x+y, x+2\eta y),$$
and
$$
\left\{
\begin{aligned}
&\gamma=\eta={\rm sign}(\lambda^+_1+\lambda^+_2) \qquad &&{\rm when} \quad (\lambda^+_1+\lambda^+_2)(\lambda^-_1+\lambda^-_2)>0,\\
&\gamma=-\eta=1 \qquad &&{\rm when} \quad (\lambda^+_1+\lambda^+_2)(\lambda^-_1+\lambda^-_2)<0.
\end{aligned}
\right.
$$
\label{nn}
\end{lm}

\begin{proof}
For $Z\in\Omega_{nn}$ we know that $O$ is a node of both $X$ and $Y$ with two different eigenvalues by \cite[Theorem 4.3]{ZZF}. Moreover, using the change $(x, y)\rightarrow(-x, y)$ it is enough to consider $Z\in\Omega_{nn}$ satisfying (\ref{onecase}). In this case, according to the dynamics on $\Sigma$ given in Lemma~\ref{slidy}, we get four local phase portraits of $Z$ near $O$ as shown in Figure~\ref{nnn}, depending on the sign of $\lambda^\pm_1+\lambda^\pm_2$, namely the stability of $O$ as an equilibrium of $X$ and $Y$. However, we notice that the phase portrait (d) of Figure~\ref{nnn} can be transformed into
(b) of Figure~\ref{nnn} by the change $(x, y)\rightarrow(-x, -y)$, so that there are essentially three different types of the local phase portraits of $Z$ near $O$. Besides, a simple analysis implies that the phase portrait of $Z_{nn}$ is (NN-1) (resp. (NN-2) and (NN-3))
of Figure~\ref{localphaseportraits} if $\gamma=\eta=1$ (resp. $\gamma=-\eta=1$ and $\gamma=\eta=-1$).
\begin{figure}
  \begin{minipage}[t]{0.5\linewidth}
  \centering
  \includegraphics[width=1.7in]{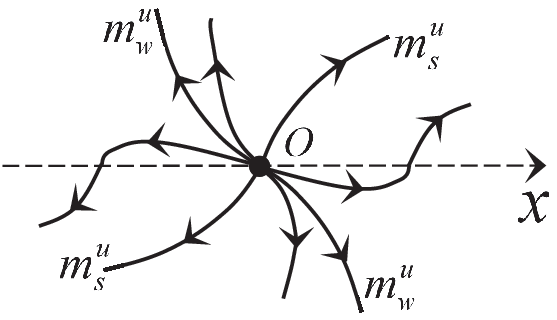}
  \caption*{(a)~ {\small $\lambda^+_1+\lambda^+_2>0,~\lambda^-_1+\lambda^-_2>0$}}
  \end{minipage}
  \begin{minipage}[t]{0.5\linewidth}
  \centering
  \includegraphics[width=1.7in]{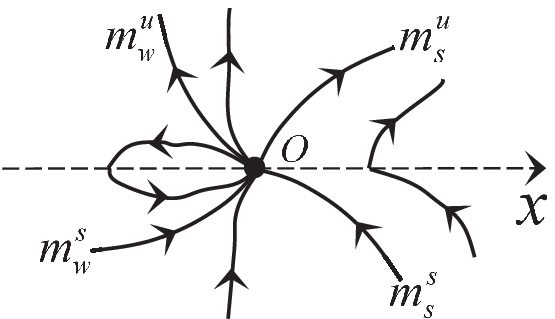}
  \caption*{(b)~ {\small $\lambda^+_1+\lambda^+_2>0,~\lambda^-_1+\lambda^-_2<0$}}
  \end{minipage}
  \begin{minipage}[t]{0.5\linewidth}
  \centering
  \includegraphics[width=1.7in]{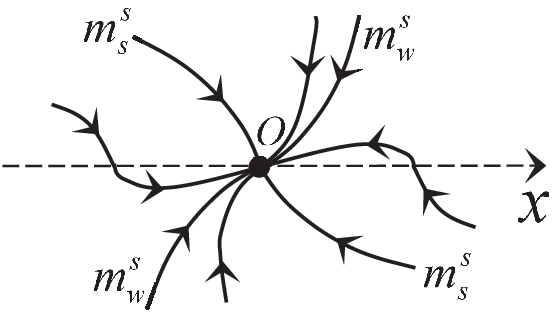}
  \caption*{(c)~ {\small $\lambda^+_1+\lambda^+_2<0,~\lambda^-_1+\lambda^-_2<0$}}
  \end{minipage}
    \begin{minipage}[t]{0.5\linewidth}
  \centering
  \includegraphics[width=1.7in]{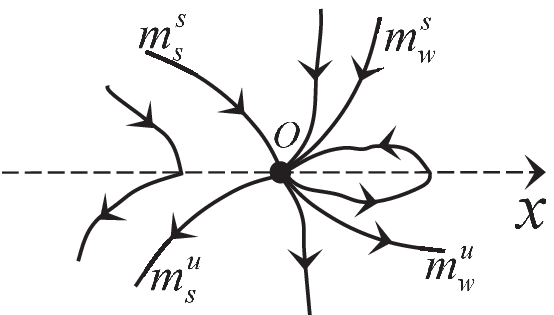}
  \caption*{(d)~ {\small $\lambda^+_1+\lambda^+_2<0,~\lambda^-_1+\lambda^-_2>0$}}
  \end{minipage}
\caption{{\small Local phase portraits of $Z\in\Omega_{nn}$ satisfying (\ref{onecase}) near $O$.}}
\label{nnn}
\end{figure}

The homeomorphism between $Z\in\Omega_{nn}$ and $Z_{nn}$ can be constructed by a similar method to the proofs of foregoing lemmas.
In fact, consider the case of $\lambda^+_1+\lambda^+_2>0, \lambda^-_1+\lambda^-_2>0$ and $\gamma=\eta=1$ as an example. We can choose
two sufficiently small neighborhoods $U\subset {\mathcal U}_0$ and $V\subset {\mathcal U}_0$ of $O$ such that
$Z$ is transverse to the boundary of $U$ and $Z_{nn}$ is transverse to the boundary of $V$. Then there is always a homeomorphism $H: \Sigma\cap U\rightarrow \Sigma\cap V$
satisfying $H(O)=O, H(\Sigma_l\cap U)=\Sigma_l\cap V$ and $H(\Sigma_r\cap U)=\Sigma_r\cap V$, where
$\Sigma_l=\{(x, 0)\in {\mathcal U}: x<0\}$ and $\Sigma_r=\{(x, 0)\in {\mathcal U}: x>0\}$.
Like the construction of $H_n$ in the proof of Lemma~\ref{fn}, we are able to extend $H$ for $\Sigma^+\cap U$ and $\Sigma^-\cap U$ respectively,
and finally obtain a homeomorphism from $U$ to $V$ that provides the $\Sigma$-equivalence between $Z\in\Omega_{nn}$ with $\lambda^+_1+\lambda^+_2>0, \lambda^-_1+\lambda^-_2>0$ and $Z_{nn}$ with $\gamma=\eta=1$. That is, Lemma~\ref{nn} holds.
\end{proof}

\begin{lm}
If $Z=(X, Y)\in\Omega_{ns}$, then $Z$ is locally $\Sigma$-equivalent to $Z_{ns}=(X_{ns}, Y_{ns})\in\Omega_{ns}$ near the origin, where
$$X_{ns}(x, y)=(2\xi x+y, x+2\xi y),\qquad Y_{ns}(x, y)=(y, x),$$
and
$$\xi=
\left\{
\begin{aligned}
&{\rm sign}(\lambda^-_1+\lambda^-_2) \qquad {\rm when}~ \lambda^-_1\lambda^-_2>0,\\
&{\rm sign}(\lambda^+_1+\lambda^+_2) \qquad {\rm when}~ \lambda^+_1\lambda^+_2>0.
\end{aligned}
\right.
$$
\label{ns}
\end{lm}

\begin{proof}
Using the changes $(x, y)\rightarrow(x, -y)$ and $(x, y)\rightarrow(-x, y)$, we only need to consider $Z\in\Omega_{ns}$ satisfying (\ref{onecase}), $\lambda^+_1\lambda^+_2>0$ and $\lambda^-_1\lambda^-_2<0$.
In this case, $O$ is a node of $X$ and a saddle of $Y$ by \cite[Theorems 4.3, 4.4]{ZZF}.
Combining with the dynamics on $\Sigma$ given in Lemma~\ref{slidy}, we get two different types of the local phase portraits of $Z$ near $O$ as shown in Figure~\ref{nsn}, depending on the sign of $\lambda^+_1+\lambda^+_2$. Regarding $Z_{ns}$, its phase portrait is (NS-1) (resp. (NS-2)) of Figure~\ref{localphaseportraits} if $\xi=1$ (resp. $\xi=-1$).

\begin{figure}
  \begin{minipage}[t]{0.5\linewidth}
  \centering
  \includegraphics[width=1.75in]{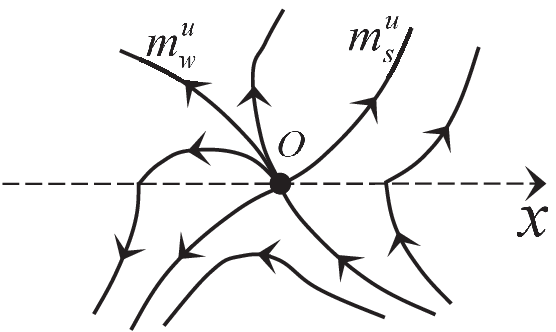}
  \caption*{(a)~ {\small $\lambda^+_1+\lambda^+_2>0$}}
  \end{minipage}
  \begin{minipage}[t]{0.5\linewidth}
  \centering
  \includegraphics[width=1.75in]{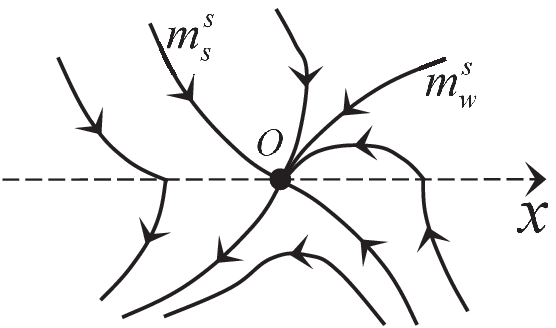}
  \caption*{(b)~ {\small $\lambda^+_1+\lambda^+_2<0$}}
  \end{minipage}
\caption{{\small Local phase portraits of $Z\in\Omega_{ns}$ satisfying (\ref{onecase}), $\lambda^+_1\lambda^+_2>0$ and $\lambda^-_1\lambda^-_2<0$ near $O$.}}
\label{nsn}
\end{figure}

Consider two sufficiently small neighborhoods $U\subset {\mathcal U}_0$ and $V\subset {\mathcal U}_0$ of $O$ such that
$Z$ is transverse to the boundary of $U$ and $Z_{ns}$ is transverse to the boundary of $V$. For each one of the above two cases, we can define a homeomorphism $H$ with $H(O)=O$ to identify $\Sigma\cap U$ with $\Sigma\cap V$ by the arc length parametrization. Then $H$ can be extended for $\Sigma^+\cap U$ (resp. $\Sigma^-\cap U$) as the construction of $H_n$ (resp. $H_s$) in the proof of Lemma~\ref{fn} (resp. Lemma~\ref{fsfsfs}). That is, $H$ is a homeomorphism from $U$ to $V$ that provides $\Sigma$-equivalence, and then Lemma~\ref{ns} holds.
\end{proof}

\begin{lm}
If $Z=(X, Y)\in\Omega_{ss}$, then $Z$ is locally $\Sigma$-equivalent to $Z_{ss}=(X_{ss}, Y_{ss})\in\Omega_{ss}$ near the origin, where
$$X_{ss}(x, y)=(y, x),\qquad Y_{ss}(x, y)=(y, x).$$
\label{ssss}
\end{lm}

\begin{proof}
For $Z\in\Omega_{ss}$ we know that $O$ is a saddle of both $X$ and $Y$ by \cite[Theorem 4.4]{ZZF}.
Using the change $(x, y)\rightarrow(-x, y)$ we only need to consider $Z\in\Omega_{ss}$ satisfying (\ref{onecase}). Together with
the dynamics on $\Sigma$ given in Lemma~\ref{slidy}, this implies that the local phase portrait of $Z$ near $O$ is as shown in Figure~\ref{ssn}.
Moreover, the phase portrait of $Z_{ss}$ is (SS) of Figure~\ref{localphaseportraits}.
\begin{figure}
  \begin{minipage}[t]{1.0\linewidth}
  \centering
  \includegraphics[width=1.8in]{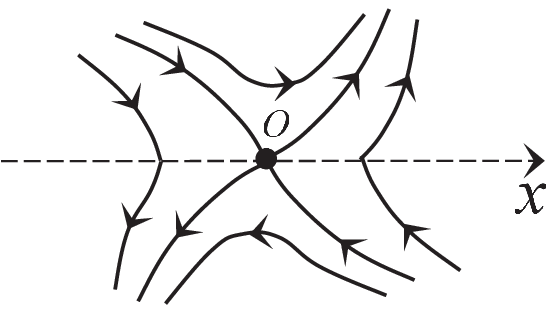}
  \end{minipage}
\caption{{\small Local phase portrait of $Z\in\Omega_{ss}$ satisfying (\ref{onecase}) near $O$.}}
\label{ssn}
\end{figure}

Consider two sufficiently small neighborhoods $U\subset {\mathcal U}_0$ and $V\subset {\mathcal U}_0$ of $O$ such that
$Z$ is transverse to the boundary of $U$ and $Z_{ss}$ is transverse to the boundary of $V$ . We can define a homeomorphism $H$ with $H(O)=O$ to identify $\Sigma\cap U$ with $\Sigma\cap V$ by the arc length parametrization. Repeating the construction of $H_s$ in the proof of Lemma~\ref{fsfsfs}, we extend $H$ for $\Sigma^+\cap U$ and $\Sigma^-\cap U$ respectively, and finally obtain a homeomorphism from $U$ to $V$ that provides $\Sigma$-equivalence between $Z\in\Omega_{ss}$ and $Z_{ss}$. This proves Lemma~\ref{ssss}.
\end{proof}

Now we are in a suitable position to prove Theorem~\ref{normalform}.

\begin{proof}[{\bf Proof of Theorem~\ref{normalform}}]
For $Z\in\Omega_{ff}$ (resp. $\Omega_{fn}, \Omega_{fs}, \Omega_{nn}, \Omega_{ns}, \Omega_{ss}$), the corresponding piecewise linear vector field $Z_L$ given in (\ref{pwl}) is also in $\Omega_{ff}$ (resp. $\Omega_{fn}, \Omega_{fs}, \Omega_{nn}, \Omega_{ns}, \Omega_{ss}$). Thus, by Lemmas~\ref{ff}-\ref{ssss} both $Z$ and $Z_L$ are locally $\Sigma$-equivalent to $Z_{ff}$ (resp. $Z_{fn}, Z_{fs}, Z_{nn}, Z_{ns}, Z_{ss}$) near $O$, which implies that
$Z$ is locally $\Sigma$-equivalent to $Z_L$ near $O$ if $Z\in\Omega_{ff}$ (resp. $\Omega_{fn}, \Omega_{fs}, \Omega_{nn}, \Omega_{ns}, \Omega_{ss}$).
Since $\Omega_1=\Omega_{ff}\cup\Omega_{fn}\cup\Omega_{fs}\cup\Omega_{nn}\cup\Omega_{ns}\cup\Omega_{ss}$, $Z$ is locally $\Sigma$-equivalent to $Z_L$ near $O$ for every $Z\in\Omega_1$. Collecting all non-equivalent phase portraits of $Z_{ff}$, $Z_{fn}, Z_{fs}, Z_{nn}, Z_{ns}$ and $Z_{ss}$ obtained in Lemmas \ref{ff}-\ref{ssss}, we get 11 local phase portraits of $Z\in\Omega_1$ near $O$ as shown in Figure~\ref{localphaseportraits}.
\end{proof}

From Lemmas~\ref{ff}, \ref{nn} and \ref{ssss} we find that some $Z\in\Omega_1$ are locally $\Sigma$-equivalent to smooth linear vector fields
near the origin.

As indicated in Section 1, Theorem~\ref{normalform} does not allow the same eigenvalue for the Jacobian matrices $A^+$ and $A^-$ respectively in order that the vector field in $\Omega_0$ is locally $\Sigma$-equivalent to its linear part near the origin.
The next proposition provides an example showing that the vector field in $\Omega_0$ might not be locally $\Sigma$-equivalent to its linear part near the origin if the Jacobian matrix $A^+$ or $A^-$ has the same eigenvalue.

\begin{prop}
Consider the piecewise smooth vector field $Z^*=(X^*, Y^*)$ with
$$X^*(x, y)=(y, x),~~~~~~~~~~Y^*(x, y)=\left(x+\frac{1}{2}\Gamma\left(x, y\right),~x+y+\frac{1}{2}\Gamma\left(x, y\right)\right),$$
where
$$\Gamma\left(x, y\right)=
\left\{
\begin{aligned}
\left(x^2+y^2\right)^{1/2}&\left(-\frac{1}{2}\ln\left(x^2+y^2\right)\right)^{-3/2}~~~~~&&{\rm if}~~x^2+y^2<1,\\
&0~~~~~&&{\rm if}~~x^2+y^2=0.
\end{aligned}
\right.$$
Then $Z^*\in\Omega_0$ and it is not locally $\Sigma$-equivalent to its linear part
$Z^*_L=(X^*_L, Y^*_L)$ near the origin, where $X^*_L(x, y)=(y, x)$ and $Y^*_L(x, y)=(x,~x+y)$.
\label{example}
\end{prop}

\begin{proof}
We start by proving $Z^*=(X^*, Y^*)\in\Omega_0$. In fact, a straightway calculation implies $\Gamma(0, 0)=0$,
$\Gamma_x(0, 0)=\Gamma_y(0, 0)=0$ and $\Gamma(x, y)$ is continuously differential near $O$.
Thus $Y^*$ is a $\mathcal{C}^1$ vector filed having $O$ as a non-degenerate equilibrium, i.e.,
$Y^*(0, 0)=(0, 0)$ and the determinant of the Jacobian matrix of $Y^*$ at $O$ is nonzero.
Clearly, the vector field $X^*$ is also $\mathcal{C}^1$ and $O$ is a linear saddle of it.
Accordingly, condition (\ref{adc}) holds for $Z^*$. On the other hand, we have $X^*_{2x}(0, 0)=Y^*_{2x}(0, 0)=1$, so that (\ref{condi}) also holds for $Z^*$, where $X^*_2$ and $Y^*_2$ are the ordinates of $X^*$ and $Y^*$ respectively. In conclusion, we get $Z^*\in\Omega_0$
from the definition of $\Omega_0$ given above (\ref{adc}), and the linear part of $Z^*$ is $Z^*_L$ from $\Gamma_x(0, 0)=\Gamma_y(0, 0)=0$.

Next we determine the local phase portraits of $Z^*$ and $Z^*_L$ near $O$ in order to prove that $Z^*$ is not locally
$\Sigma$-equivalent to $Z^*_L$. Regarding $Z^*$, $O$ is a saddle of $X^*$ with the unstable manifold $\{(x, y)\in\mathbb{R}^2: y=x, y>0\}$ and
the stable manifold $\{(x, y)\in\mathbb{R}^2: y=-x, y>0\}$. From \cite[Example 4.3]{ZZF} we have known that all orbits of $Y^*$ near $O$
starting from the negative $x$-axis enter into $\Sigma^-$ and then reach the positive $x$-axis after a finite time. Thus the local phase portrait of $Z^*$ near $O$ is as shown in Figure~\ref{fsexam}(a). Regarding $Z^*_L$, $O$ is an unstable non-diagonalizable node of $Y^*_L$ with the characteristic direction $x=0$. Due to $X^*_L=X^*$, we conclude that the phase portrait of $Z^*_L$ is as shown in Figure~\ref{fsexam}(b).

Consider the orbits of $Z^*$ and $Z^*_L$ starting from the negative $x$-axis. From Figure~\ref{fsexam} we observe that these orbits of $Z^*$ intersect the positive $x$-axis, but the ones of $Z^*_L$ do not. Since any $\Sigma$-equivalence sends the orbits of $Z^*$ to the orbits of $Z^*_L$, preserving
the switching line $\Sigma$, it also preserves the intersections between the orbits and $\Sigma$. Consequently, $Z^*$ cannot be locally $\Sigma$-equivalent to $Z^*_L$ near $O$.
\begin{figure}
  \begin{minipage}[t]{0.5\linewidth}
  \centering
  \includegraphics[width=1.7in]{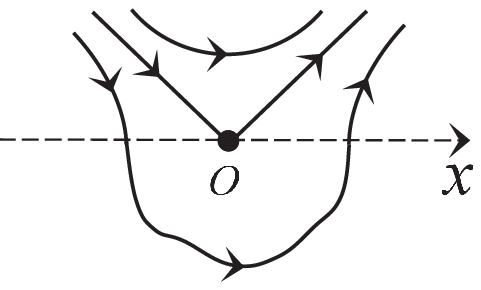}
  \caption*{(a)}
  \end{minipage}
  \begin{minipage}[t]{0.5\linewidth}
  \centering
  \includegraphics[width=1.7in]{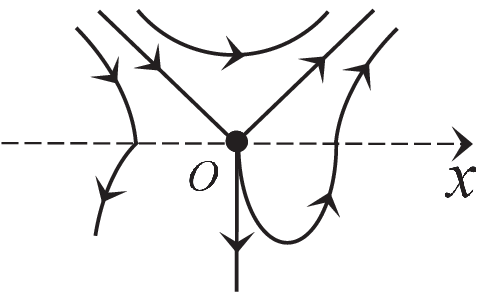}
  \caption*{(b)}
  \end{minipage}
  \caption{{\small Local phase portraits of $Z^*$ and $Z^*_L$ near $O$, where (a) is for $Z^*$ and (b) is for $Z^*_L$.}}
\label{fsexam}
\end{figure}
\end{proof}

\section{Proof of Theorem~\ref{stability}}
\setcounter{equation}{0}
\setcounter{lm}{0}
\setcounter{thm}{0}
\setcounter{rmk}{0}
\setcounter{df}{0}
\setcounter{cor}{0}

Since $\Omega_{ff}$ is an open set of $\Omega_1\subset\Omega_0$, any small perturbation of $Z\in\Omega_{ff}$ inside $\Omega_0$
belongs to $\Omega_{ff}$. In particular,
the value of sign function $\alpha$ defined in Lemma~\ref{ff} is the same for $Z\in\Omega_{ff}$ and its perturbation. Thus by Lemma~\ref{ff} both $Z\in\Omega_{ff}$ and any perturbation of it inside $\Omega_0$ are locally $\Sigma$-equivalent to the same normal form $Z_{ff}$ near $O$. This means that $Z\in\Omega_{ff}$ is locally $\Sigma$-structurally stable with respect to $\Omega_0$ near $O$.
Similar argument can be applied to $Z$ belonging to $\Omega_{fn}$, $\Omega_{fs}, \Omega_{nn}, \Omega_{ns}$ and $\Omega_{ss}$ respectively. Finally,
due to $\Omega_1=\Omega_{ff}\cup\Omega_{fn}\cup\Omega_{fs}\cup\Omega_{nn}\cup\Omega_{ns}\cup\Omega_{ss}$, we conclude that
$Z\in\Omega_1$ is locally $\Sigma$-structurally stable with respect to $\Omega_0$ near $O$, that is, the sufficiency holds.

To obtain the necessity, we can equivalently prove that $Z\in\Omega_0$ is not locally $\Sigma$-structurally stable with respect to $\Omega_0$ near $O$ if $Z\in\Omega_0\setminus\Omega_1$. To do this, we classify $\Omega_0\setminus\Omega_1$ into two subsets:
\begin{eqnarray}
\Omega_2=\{Z\in\Omega_0: {\rm Im}\lambda^+_1{\rm Im}\lambda^-_1\ne0, \ell=0\}, \qquad
\Omega_3=\{Z\in\Omega_0: (\lambda^+_1-\lambda^+_2)(\lambda^-_1-\lambda^-_2)=0\}.
\label{subsetdefinitions}
\end{eqnarray}
Clearly, $\Omega_0\setminus\Omega_1=\Omega_2\cup\Omega_3$. If $Z=(X, Y)\in\Omega_2$, then $O$ is an equilibrium of focus type for both $X$ and $Y$. In this case, $O$ is a non-smooth center or pseudo-focus of focus-focus type of $Z$ with the first Lyapunov constant $\ell=0$ as it was clarified in
\cite[Theorem B]{CGP}. Since $\ell$ only depends on the linear part of $Z$, we easily obtain a perturbed vector field with $\ell>0$ and a perturbed one with $\ell<0$ by perturbing the linear part of $Z$ in $\Omega_0$. This means that, for any sufficiently small neighborhood of $Z$ in $\Omega_0$, there
always exist two vector fields where $O$ is a pseudo-focus with the different stability. Even limit cycles can bifurcate from $O$, e.g., \cite{ZK}.
Then these two perturbed vector fields are not locally $\Sigma$-equivalent near $O$, so that
$Z\in\Omega_2$ is not locally $\Sigma$-structurally stable with respect to $\Omega_0$ near $O$.

If $Z=(X, Y)\in\Omega_3$, then at least one of $\lambda^+_1=\lambda^+_2$ and $\lambda^-_1=\lambda^-_2$ holds. Without loss of generality we assume that
$\lambda^+_1=\lambda^+_2$. Writing $X$ near $O$ as
$$X=A^+\left(x, y\right)^\top+\Upsilon^+(x, y),$$
where $\Upsilon^+(x, y)$ is the higher order terms and
$$
A^+=\left(
\begin{array}{cc}
a_{11}^+&a_{12}^+\\
a_{21}^+&a_{22}^+
\end{array}\right),
$$
we get
\begin{eqnarray}
\lambda^+_1=\lambda^+_2=\frac{1}{2}(a^+_{11}+a^+_{22}), \qquad (a^+_{11}-a^+_{22})^2+4a^+_{12}a^+_{21}=0, \qquad a^+_{21}\ne0.
\label{dnjknfejfew}
\end{eqnarray}
Here $a^+_{21}\ne0$ is due to the fact that $Z\in\Omega_0$ satisfies (\ref{condi}).
Consider the vector field $Z_\varepsilon=(X_\varepsilon, Y)$ with
$$
X_\varepsilon=A^+_\varepsilon\left(x, y\right)^\top+\Upsilon^+(x, y)
$$
and
$$
A^+_\varepsilon=\left(
\begin{array}{cc}
a_{11}^+&\frac{a_{12}^+a_{21}^++\varepsilon/4}{a_{21}^++\varepsilon}\\
a_{21}^++\varepsilon&a_{22}^+
\end{array}\right).
$$
Then for any sufficiently small neighborhood of $Z$ in $\Omega_0$, there exists $\varepsilon_0>0$ such that $Z_\varepsilon$
lies in the neighborhood for all $-\varepsilon_0<\varepsilon<\varepsilon_0$. Denote the eigenvalues of $A^+_\varepsilon$ by $\lambda^+_{\varepsilon, 1}$ and $\lambda^+_{\varepsilon, 2}$. It follows from (\ref{dnjknfejfew}) that
$$\lambda^+_{\varepsilon, 1}=\lambda^+_1+\frac{\sqrt{-\varepsilon}}{2}i, \qquad \lambda^+_{\varepsilon, 2}=\lambda^+_1-\frac{\sqrt{-\varepsilon}}{2}i$$
for $-\varepsilon_0<\varepsilon<0$, while for $\varepsilon_0>\varepsilon>0$,
$$\lambda^+_{\varepsilon, 1}=\lambda^+_1+\frac{\sqrt\varepsilon}{2}, \qquad \lambda^+_{\varepsilon, 2}=\lambda^+_1-\frac{\sqrt\varepsilon}{2}.$$
In the case of $-\varepsilon_0<\varepsilon<0$, $O$ is a focus of $X_\varepsilon$ by \cite[Theorem 4.2]{ZZF}, so that
all orbits of $Z_{\varepsilon}$ near $O$ starting from the positive $x$-axis enter into $\Sigma^+$ and then reach the negative $x$-axis as $t$ increases (resp. decreases) if $a^+_{21}>0$ (resp. $<0$).
In the case of $\varepsilon_0>\varepsilon>0$, $O$ is a diagonalizable node of $X_\varepsilon$ by \cite[Theorem 4.3]{ZZF}, which has two characteristic directions with the nonzero slope due to $a^+_{21}\ne0$. Thus all orbits of $Z_{\varepsilon}$ near $O$ starting from the positive $x$-axis cannot reach the negative $x$-axis from $\Sigma^+$ as $t$ increases (resp. decreases) if $a^+_{21}>0$ (resp. $<0$). As indicated in the proof of Proposition~\ref{example}, any $\Sigma$-equivalence sends the orbits of $Z_\varepsilon$ with $-\varepsilon_0<\varepsilon<0$ to the orbits of $Z_\varepsilon$ with $\varepsilon_0>\varepsilon>0$, preserving the switching line $\Sigma$ and the intersections of $\Sigma$ and the orbits. Consequently, $Z_\varepsilon$ with $-\varepsilon_0<\varepsilon<0$ cannot be locally $\Sigma$-equivalent to $Z_\varepsilon$ with $\varepsilon_0>\varepsilon>0$ near $O$. This means that, for any sufficiently small neighborhood of $Z\in\Omega_3$ in $\Omega_0$, there are always two vector fields that are not locally $\Sigma$-equivalent near $O$. So $Z\in\Omega_3$ is not locally $\Sigma$-structurally stable with respect to $\Omega_0$ near $O$. Recalling the last paragraph, we conclude the necessity. This ends the proof of Theorem~\ref{stability}.

\section{Proof of Theorem~\ref{bifurcation}}
\setcounter{equation}{0}
\setcounter{lm}{0}
\setcounter{thm}{0}
\setcounter{rmk}{0}
\setcounter{df}{0}
\setcounter{cor}{0}

Before proving Theorem~\ref{bifurcation}, we study the limit cycle bifurcations by perturbing
the following piecewise linear vector field
\begin{eqnarray}
Z_0(x, y)=\left\{
\begin{aligned}
&X_0(x, y)=(ay, x) \qquad &&{\rm if}\quad y>0,\\
&Y_0(x, y)=(by, x) \qquad &&{\rm if}\quad y<0,
\end{aligned}
\right.
\label{ppll}
\end{eqnarray}
where $a, b\in\mathbb{R}$ satisfies $ab\ne0$.

\begin{prop}
Consider the piecewise linear vector field $Z_0$ with $ab\ne0$ given in {\rm(\ref{ppll})}. Then $Z_0\in\Omega_1$ if either $a>0$ or $b>0$, and $Z_0\in\Omega_2$ if $a<0$ and $b<0$, where $\Omega_1\subset\Omega_0$ and $\Omega_2\subset\Omega_0\setminus\Omega_1$ are defined in {\rm(\ref{subsetdefinition})} and {\rm(\ref{subsetdefinitions})} separately. Moreover, $Z_0$ has no limit cycles.
\label{zoooo}
\end{prop}

\begin{proof}
The first part of Proposition~\ref{zoooo} follows directly from the definitions of $\Omega_1$ and $\Omega_2$.
Since $O$ is saddle or center of $X_0$ and $Y_0$, it is impossible for $Z_0$ to have limit cycles totally contained in the half plane $y\ge0$ or $y\le0$. On the other hand, when $O$ is a center of both $X_0$ and $Y_0$, it is a global non-smooth center for $Z_0$,
so that $Z_0$ has no limit cycles occupying the half planes $y>0$ and $y<0$.
Clearly, there also exist no limit cycles occupying the half planes $y>0$ and $y<0$ when $O$ is a saddle of $X_0$ or $Y_0$.
\end{proof}

Next we state two bifurcation results by perturbing the piecewise linear vector field $Z_0$ given in (\ref{ppll}).

\begin{prop}
Consider the piecewise linear vector field $Z_0\in\Omega_0$ in {\rm(\ref{ppll})}, and the piecewise polynomial vector field
$$
Z^f_\epsilon(x,y)=\left\{
\begin{aligned}
&X^f_\epsilon(x,y)=((a-\epsilon)y-\epsilon, x) \qquad &&{\rm if}\quad y>0,\\
&Y^f_\epsilon(x,y)=\left((b-\epsilon)y+\epsilon, x+\epsilon\frac{\partial f(x, \epsilon)}{\partial x}\right) \qquad &&{\rm if}\quad y<0,
\end{aligned}
\right.
$$
where $\epsilon\ge0$ and
\begin{eqnarray}
f(x, \epsilon)=x\prod^m_{i=1}\left(x^2-\left(\frac{i\epsilon}{m}\right)^2\right).
\label{saasmajf}
\end{eqnarray}
 Then $Z^f_\epsilon=Z_0$ for $\epsilon=0$. Besides, for any given $a$ and $b$ satisfying $0<|a|\le1/2$ and $0<|b|\le1/2$, there exists $0<\epsilon_0<\min\{|a|, |b|\}$ such that for $0<\epsilon<\epsilon_0$, $Z^f_\epsilon$ has exactly $m$ hyperbolic crossing limit cycles $\Gamma_i$ $(i=1, 2, \cdot\cdot\cdot, m)$ bifurcating from the non-smooth equilibrium $O$ of $Z_0$, where $\Gamma_i$ obeys the algebraic curve
 $$
 \Gamma^+_i: \frac{1}{2}x^2-\frac{a-\epsilon}{2}y^2+\epsilon y=\frac{1}{2}\left(\frac{i\epsilon}{m}\right)^2$$
in the half plane $y\ge0$ and the algebraic curve
$$\Gamma^-_i: \frac{1}{2}x^2+\epsilon f(x, \epsilon)-\frac{b-\epsilon}{2}y^2-\epsilon y=\frac{1}{2}\left(\frac{i\epsilon}{m}\right)^2$$
in the half plane $y\le0$. Moreover, $\Gamma_i$ is stable if $m-i$ is even and unstable if $m-i$ is odd.
\label{ejhfjncjff}
\end{prop}

\begin{prop}
Consider the piecewise linear vector field $Z_0\in\Omega_0$ in {\rm(\ref{ppll})}, and the piecewise $\mathcal{C}^\infty$ vector field
$$
Z^g_\epsilon(x,y)=\left\{
\begin{aligned}
&X^g_\epsilon(x,y)=((a-\epsilon)y-\epsilon, x) \qquad &&{\rm if}\quad y>0,\\
&Y^g_\epsilon(x,y)=\left((b-\epsilon)y+\epsilon, x+\epsilon\frac{\partial g(x, \epsilon)}{\partial x}\right) \qquad &&{\rm if}\quad y<0,
\end{aligned}
\right.
$$
where $\epsilon\ge0$ and $g(x, \epsilon)$ is a $\mathcal{C}^\infty$ function given by
$$g(x, \epsilon)=\left\{
\begin{aligned}
&0 \qquad &&{\rm if}\quad x\le0,\\
&e^{-1/x}\sin\left(\frac{\pi\epsilon}{x}\right) \qquad &&{\rm if}\quad x>0.
\end{aligned}
\right.
$$
Then $Z^g_\epsilon=Z_0$ for $\epsilon=0$. Besides, for any given $a$ and $b$ satisfying $0<|a|\le1/2$ and $0<|b|\le1/2$, there exists $0<\epsilon_0<\min\{|a|, |b|\}$ such that for $0<\epsilon<\epsilon_0$, $Z^g_\epsilon$ has infinitely many hyperbolic crossing limit cycles $\Theta_i$ $(i\in\mathbb{N}^+)$ bifurcating from the non-smooth equilibrium $O$ of $Z_0$, where $\Theta_i$ obeys the algebraic curve
$$\Theta^+_i: \frac{1}{2}x^2-\frac{a-\epsilon}{2}y^2+\epsilon y=\frac{1}{2}\left(\frac{\epsilon}{i}\right)^2$$
in the half plane $y\ge0$ and the algebraic curve
$$\Theta^-_i: \frac{1}{2}x^2+\epsilon g(x, \epsilon)-\frac{b-\epsilon}{2}y^2-\epsilon y=\frac{1}{2}\left(\frac{\epsilon}{i}\right)^2$$
in the half plane $y\le0$. Moreover, $\Theta_i$ is stable if $i$ is odd and unstable if $i$ is even.
\label{ejhadafjff}
\end{prop}

Propositions~\ref{ejhfjncjff} and \ref{ejhadafjff} will be proved later on. If $a=b>0$ (resp. $<0$), then $Z_0=X_0=Y_0$ is a linear vector field having $O$ as a saddle (resp. center). Thus our results reveal that any finitely or infinitely many limit cycles can bifurcate from some linear saddle and center under non-smooth perturbations. Besides, observe that $Z^f_\epsilon$ and $Z^g_\epsilon$ are both piecewise smooth Hamiltonian systems. This means that it is possible for piecewise smooth Hamiltonian systems to have limit cycles, but this cannot occur
in smooth Hamiltonian systems as well known.

Now we are in a position to provide the proof of Theorem~\ref{bifurcation}.

\begin{proof}[{\bf Proof of Theorem~\ref{bifurcation}}]
For $Z=(X, Y)\in\Omega_0$ we consider the three-parametric perturbed vector field $Z^{\bm\epsilon}=(X^{\bm{\epsilon}}, Y^{\bm{\epsilon}})$ with
$$
\begin{aligned}
X^{\bm{\epsilon}}(x, y)&=\left(X_1(x, y)-X_{2x}(0, 0)\epsilon_1+X_{2x}(0, 0)\epsilon_1x, X_2(x, y)\right),\\
Y^{\bm\epsilon}(x, y)&=\left(Y_1(x, y)+Y_{2x}(0, 0)\epsilon_1+Y_{2x}(0, 0)\epsilon_1x+\epsilon_2Y_2(x, y), Y_2(x, y)+\epsilon_3\right),
\end{aligned}
$$
where $\bm\epsilon=(\epsilon_1, \epsilon_2, \epsilon_3)\in\mathbb{R}^3$ is a parameter vector. Clearly, $Z^{\bm{\epsilon}}=Z$ for ${\bm{\epsilon}}=(0, 0, 0)$, and $Z^{\bm{\epsilon}}\in\Omega$. We claim that for any small neighborhood of $\bm\epsilon=(0, 0, 0)$ there always exists
$\bm\epsilon_0$ in the neighborhood such that $Z^{\bm\epsilon_0}$ has a crossing limit cycle bifurcating from the non-smooth equilibrium $O$ of $Z$.
In fact, fixing $\epsilon_2=\epsilon_3=0$ we have
\begin{equation}\label{cmkc}
\begin{aligned}
&X^{\bm\epsilon}_1(0, 0)=-X_{2x}(0, 0)\epsilon_1,\qquad &&X^{\bm\epsilon}_2(0, 0)=0,\qquad &&X^{\bm\epsilon}_{2x}(0, 0)=X_{2x}(0, 0),\\
&Y^{\bm\epsilon}_1(0, 0)=Y_{2x}(0, 0)\epsilon_1,\qquad &&Y^{\bm\epsilon}_2(0, 0)=0,\qquad &&Y^{\bm\epsilon}_{2x}(0, 0)=Y_{2x}(0, 0),
\end{aligned}
\end{equation}
where $(X^{\bm\epsilon}_1, X^{\bm\epsilon}_2)$ and $(Y^{\bm\epsilon}_1, Y^{\bm\epsilon}_2)$ are the coordinates of $X^{\bm{\epsilon}}$ and $Y^{\bm{\epsilon}}$ respectively. So $O$ is an invisible-invisible fold-fold point of $Z^{\bm\epsilon}$ for $\epsilon_1>0$ and $\epsilon_2=\epsilon_3=0$. Besides, all orbits of $Z^{\bm\epsilon}$ near $O$ turn around $O$ because $X^{\bm\epsilon}_{2x}(0, 0)Y^{\bm\epsilon}_{2x}(0, 0)=X_{2x}(0, 0)Y_{2x}(0, 0)>0$. Here $X_{2x}(0, 0)Y_{2x}(0, 0)>0$ is due that $Z\in\Omega_0$ satisfies (\ref{condi}). Thus $O$ is either a non-smooth center or a pseudo-focus of $Z^{\bm\epsilon}$ for $\epsilon_1>0$ and $\epsilon_2=\epsilon_3=0$. By the time reversal, without loss of generality next we only work with the case where all orbits of $Z^{\bm\epsilon}$ near $O$ rotate counterclockwise, namely $X_{2x}(0, 0)>0$ and $Y_{2x}(0, 0)>0$.

If $O$ is a stable (resp. unstable) pseudo-focus of $Z^{\bm\epsilon}$ for $\epsilon_1>0$ and $\epsilon_2=\epsilon_3=0$, a direct application of Proposition~\ref{pseudohopf} yields that for given $\epsilon_1>0$ and $\epsilon_2=0$ there exists $\hat\epsilon_3=\hat\epsilon_3(\epsilon_1)>0$ such that $Z^{\bm\epsilon}$ with $\epsilon_1>0$, $\epsilon_2=0$ and $-\hat\epsilon_3<\epsilon_3<0$ (resp. $0<\epsilon_3<\hat\epsilon_3$)
admits a stable (resp. unstable) crossing limit cycle bifurcating from $O$. Thus, for any small neighborhood of $\bm\epsilon=(0, 0, 0)$ we can choose some $\bm\epsilon_0=(\epsilon_{10}, \epsilon_{20}, \epsilon_{30})$ satisfying $\epsilon_{10}>0$, $\epsilon_{20}=0$ and $0<|\epsilon_{30}|<\hat\epsilon_3(\epsilon_{10})$ such that $Z^{\bm\epsilon_0}$ has a crossing limit cycle bifurcating from $O$, that is, the claim holds in the case that $O$ is a pseudo-focus.

If $O$ is a non-smooth center of $Z^{\bm\epsilon}$ for $\epsilon_1>0$ and $\epsilon_2=\epsilon_3=0$, we can obtain an upper Poincar\'e map $P_U$ near $O$ which maps a point $(x_0, 0)$ with $x_0>0$ to a point $(x_1, 0)$ with $x_1<0$, and a lower Poincar\'e map $P_L$ near $O$ which maps $(x_1, 0)$
to $(x_0, 0)$. When $\epsilon_3=0$ and $\epsilon_2$ is perturbed to be $\epsilon_2\ne0$, it is easily verify that (\ref{cmkc}) still holds, i.e.,
$O$ is still an invisible-invisible fold-fold point. In this case, we also can define an upper Poincar\'e map $\tilde P_U$ near $O$ which maps a point $(x_0, 0)$ with $x_0>0$ to a point $(x_1, 0)$ with $x_1<0$, and a lower Poincar\'e map $\tilde P_L$ near $O$ which maps $(x_1, 0)$ to a point $(x_2, 0)$ with $x_2>0$. Clearly, $P_U=\tilde P_U$ because $X^{\bm\epsilon}$ is independent of $\epsilon_2$. Moreover, we can prove that $x_2>x_0$ if $\epsilon_2>0$. In fact,
considering the vector field $Y^{\bm\epsilon}$ we define the following two equations
\begin{equation}\label{ineqcanf}
\frac{dy}{dx}=\varphi_1(x, y):=\frac{Y_2(x, y)}{Y_1(x, y)+Y_{2x}(0, 0)\epsilon_1+Y_{2x}(0, 0)\epsilon_1x}
\end{equation}
for $\epsilon_2=\epsilon_3=0$, and
\begin{equation}\label{afineqcanf}
\frac{dy}{dx}=\varphi_2(x, y):=\frac{Y_2(x, y)}{Y_1(x, y)+Y_{2x}(0, 0)\epsilon_1+Y_{2x}(0, 0)\epsilon_1x+\epsilon_2Y_2(x, y)}
\end{equation}
for $\epsilon_2\ne0$ and $\epsilon_3=0$. Since $Y^{\bm\epsilon}_1(0, 0)=Y_{2x}(0, 0)\epsilon_1>0$, the denominators of $\varphi_1(x, y)$ and $\varphi_2(x, y)$ are positive in a sufficiently small neighborhood of $O$. Thus $\varphi_1(x, y)\ge\varphi_2(x, y)$ for $\epsilon_2>0$, and the equality holds only for $(x, y)=(0, 0)$. Applying the theory of differential inequality to equations (\ref{ineqcanf}) and (\ref{afineqcanf}), we obtain the
solution of equation (\ref{ineqcanf}) with the initial value $(x_1, 0)$ always lies above the
solution of equation (\ref{afineqcanf}) with the initial value $(x_1, 0)$ in the half plane $y\le0$. So $x_2>x_0$ if $\epsilon_2>0$, and then
$O$ is an unstable pseudo-focus of $Z^{\bm\epsilon}$ for $\epsilon_1>0$, $\epsilon_2>0$ and $\epsilon_3=0$. Repeating the analysis in the last paragraph and using Proposition~\ref{pseudohopf}, for any small neighborhood of $\bm\epsilon=(0, 0, 0)$ we can choose some $\bm\epsilon_0=(\epsilon_{10}, \epsilon_{20}, \epsilon_{30})$ satisfying $\epsilon_{10}>0$, $\epsilon_{20}>0$ and $0<\epsilon_{30}<\hat\epsilon_3(\epsilon_{10}, \epsilon_{20})$ such that $Z^{\bm\epsilon_0}$ has a crossing limit cycle bifurcating from $O$, that is, the claim also holds in the case that $O$ is a non-smooth center. This, together with the last paragraph, concludes statement (1) because $Z^{\bm\epsilon}\rightarrow Z$ as $\bm\epsilon\rightarrow0$.

Let $Z_0$ be the piecewise linear vector field given in (\ref{ppll}). Then $Z_0\in\Omega_1$ if either $a>0$ or $b>0$, and $Z_0\in\Omega_0\setminus\Omega_1$ if $a<0$ and $b<0$ as indicated in Proposition~\ref{zoooo}. Thus statement (2) is a direct conclusion of Propositions~\ref{ejhfjncjff} and \ref{ejhadafjff} because $Z^f_\epsilon\rightarrow Z_0$ and $Z^g_\epsilon\rightarrow Z_0$ as $\epsilon\rightarrow0$.
\end{proof}

As well known, it is a challenge objective to establish the bifurcation diagram for some bifurcations, particularly for the higher codimension bifurcations, since a higher codimension bifurcation usually consists of too many lower codimension ones. Speaking of bifurcation diagrams, we can obtain an important information from the proof of Theorem~\ref{bifurcation}, that is, the bifurcation diagram of any vector field in $\Omega_0$ must contain a bifurcation boundary where the codimension one pseudo-Hopf bifurcation occurs. A complete bifurcation diagram of the vector fields in $\Omega_0$ will be left as a future work. Actually, this is an extremely complex work, since there exist many possible local phase portraits for the unperturbed vector fields as seen in Theorem~\ref{normalform}, and such a bifurcation has the higher codimension.

Finally, we give the proofs of Propositions~\ref{ejhfjncjff} and \ref{ejhadafjff}.

\begin{proof}[{\bf Proof of Proposition~\ref{ejhfjncjff}}]
Clearly, $Z^f_\epsilon=Z_0$ for $\epsilon=0$. The rest of this proof is completed by the following four steps.

{\it Step 1. The upper Poincar\'e map $P_U$.}
Because of $a\ne0$, we can choose $\epsilon_1>0$ such that ${\rm sign}a={\rm sign}(a-\epsilon)$ for $0<\epsilon<\epsilon_1$. In this case, $X^f_\epsilon$ has a unique equilibrium $E_X:=(0, \epsilon/(a-\epsilon))$, which is a linear center if $a-\epsilon<0$ and a linear saddle if $a-\epsilon>0$.

When $E_X$ is linear center, i.e., $a-\epsilon<0$, it lies in the lower half plane $y<0$ because of $0<\epsilon<\epsilon_1$, and then it is not a real equilibrium for $Z_\epsilon$. From the center dynamics and the direction of the vector field $X^f_\epsilon$ on the $x$-axis, it follows that the orbit of $X^f_\epsilon$ with $0<\epsilon<\epsilon_1$ starting from $(x_0, 0)$ with $x_0>0$ enters into $y>0$, and reaches again the $x$-axis at a point $(x_1, 0)$ with $x_1<0$ as $t$ increases.

When $E_X$ is a linear saddle, i.e., $a-\epsilon>0$, it lies in the upper half plane $y>0$ because of $0<\epsilon<\epsilon_1$,
and the stable and unstable manifolds of it lie in
$$\left\{(x, y)\in\mathbb{R}^2\: x\ne0, y=-\frac{x}{\sqrt{a-\epsilon}}+\frac{\epsilon}{a-\epsilon}\right\}, \qquad
\left\{(x, y)\in\mathbb{R}^2\: x\ne0, y=\frac{x}{\sqrt{a-\epsilon}}+\frac{\epsilon}{a-\epsilon}\right\},$$
respectively. Thus the stable manifold intersects the $x$-axis at $(\epsilon/\sqrt{a-\epsilon} ,0)$, and the unstable manifold intersects
the $x$-axis at $(-\epsilon/\sqrt{a-\epsilon} ,0)$. Together with the direction of the vector field $X^f_\epsilon$ on $\{(x, 0)\in\mathbb{R}^2: -\epsilon/\sqrt{a-\epsilon}<x<\epsilon/\sqrt{a-\epsilon}\}$, we get that the orbit of $X^f_\epsilon$ with $0<\epsilon<\epsilon_1$ starting from $(x_0, 0)$ with $0<x_0<\epsilon/\sqrt{a-\epsilon}$ enters into $y>0$ and reaches again the $x$-axis at a point $(x_1, 0)$ with $-\epsilon/\sqrt{a-\epsilon}<x_1<0$ as $t$ increases.

According to the last two paragraphs, we can construct an upper Poincar\'e map $P_U$ as $x_1=P_U(x_0, \epsilon)$, which is defined for $0<x_0<\varpi_u(\epsilon)$ and $0<\epsilon<\epsilon_1$, where
\begin{eqnarray}
\varpi_u(\epsilon)=
\left\{
\begin{aligned}
&+\infty \qquad &&{\rm when~} E_X~ {\rm is~ a~ center,~ i.e.,}~ a-\epsilon<0,\\
&\epsilon/\sqrt{a-\epsilon}\qquad &&{\rm when~ }E_X~ {\rm is~ a~ saddle,~ i.e.,}~ a-\epsilon>0.
\end{aligned}
\right.
\label{rutireuiv}
\end{eqnarray}
Furthermore,
calculating the first integral $H^f_X$ of $X^f_\epsilon$ we get
$$
H^f_X(x, y)=\frac{1}{2}x^2-\frac{a-\epsilon}{2}y^2+\epsilon y,
$$
so that $P_U(x_0, \epsilon)$ satisfies $H^f_X(x_0, 0)=H^f_X(P_U(x_0), 0)$, i.e.,
\begin{eqnarray}
P_U(x_0, \epsilon)=-x_0\qquad {\rm for}\quad 0<x_0<\varpi_u(\epsilon)~~{\rm and}~~0<\epsilon<\epsilon_1.
\label{PUM}
\end{eqnarray}

{\it Step 2. The lower Poincar\'e map $P_L$.} Since $b\ne0$, there exists $\epsilon_2>0$ such that ${\rm sign}b={\rm sign}(b-\epsilon)$ for $0<\epsilon<\epsilon_2$. Throughout this step, $\epsilon_2$ can be reduced if necessary.
Consider the function
$$
F(x, \epsilon)=x+\epsilon\frac{\partial f(x, \epsilon)}{\partial x}.
$$
Due to $F(0, 0)=0$ and $F_x(0,0)=1$, by the Implicit Function Theorem there exists a function $x(\epsilon)$ defined for $0<\epsilon<\epsilon_2$ such that $x(0)=0$ and $F(x(\epsilon), \epsilon)=0$. In addition, $x(\epsilon)$ is given by
\begin{eqnarray}
x(\epsilon)=(-1)^{m+1}\epsilon\prod^m_{i=1}\left(\frac{i\epsilon}{m}\right)^2+\mathcal{O}(\epsilon^{2m+2})=(-1)^{m+1}\frac{(m!)^2}{m^{2m}}\epsilon^{2m+1}
+\mathcal{O}(\epsilon^{2m+2}).
\label{erieief}
\end{eqnarray}
By the definition of invisible fold point, $(x(\epsilon), 0)$ is an invisible fold point of $Y^f_\epsilon$ for $0<\epsilon<\epsilon_2$. Combining the direction of $Y^f_\epsilon$ on the $x$-axis, we know that the orbit of $Y^f_\epsilon$ near $(x(\epsilon), 0)$ starting from a point $(x_1, 0)$ with $x_1<x(\epsilon)$ evolves in $y<0$ until it reaches the $x$-axis at a point $(x_2, 0)$ with $x_2>x(\epsilon)$ again. In this case, we can define a lower Poincar\'e map $P_L$ as $x_2=P_L(x_1, \epsilon)$ for $x_1<x(\epsilon)$ closed to $x(\epsilon)$ and $0<\epsilon<\epsilon_2$. Since the first integral of $Y^f_\epsilon$ is
$$
H^f_Y(x, y)=\frac{1}{2}x^2+\epsilon f(x, \epsilon)-\frac{b-\epsilon}{2}y^2-\epsilon y,
$$
$P_L(x_1, \epsilon)$ satisfies
\begin{eqnarray}
\frac{1}{2}x_1^2+\epsilon f(x_1, \epsilon)=\frac{1}{2}P_L(x_1, \epsilon)^2+\epsilon f(P_L(x_1, \epsilon), \epsilon).
\label{PLM}
\end{eqnarray}

Next we precisely determine the definition domain of $P_L$. Notice that $E_Y:=(x(\epsilon), -\epsilon/(b-\epsilon))$ is an equilibrium of $Y^f_\epsilon$ for $0<\epsilon<\epsilon_2$. Calculating the eigenvalues of the Jacobian matrix of $Y^f_\epsilon$ at $E_Y$, we have that $E_Y$ is of focus type if $b-\epsilon<0$ from \cite[Theorem 5.1]{ZZF}, and a saddle if $b-\epsilon>0$ from \cite[Theorem 4.4]{ZZF}.

When $E_Y$ is of focus type, i.e., $b-\epsilon<0$, it lies in the upper half plane $y>0$ because of $\epsilon>0$. Moreover, $O$ is a linear center of $Y^f_\epsilon$ for $\epsilon=0$ due to ${\rm sign}b={\rm sign}(b-\epsilon)$ for $0<\epsilon<\epsilon_2$. Thus $\epsilon_2>0$ can be reduced such that $P_L(x_1, \epsilon)$ is defined for $-1<x_1<x(\epsilon)$ and $0<\epsilon<\epsilon_2$.

When $E_Y$ is saddle, i.e., $b-\epsilon>0$, it lies in the lower half plane $y<0$ because of $\epsilon>0$. $E_Y$ has one stable (resp. unstable) manifold intersecting the $x$-axis. Let $(x_s, 0)$ (resp. $(x_u, 0)$) be the intersection between the stable (resp. unstable) manifold and the $x$-axis.
Then $H^f_Y(x_u, 0)=H^f_Y(x_s, 0)=H^f_Y(E_Y)$, i.e.,
\begin{eqnarray}
\begin{aligned}
\frac{1}{2}(x_u)^2+\epsilon f(x_u, \epsilon)=\frac{1}{2}(x_s)^2+\epsilon f(x_s, \epsilon)&=\frac{1}{2}x(\epsilon)^2+\epsilon f(x(\epsilon), \epsilon)-\frac{b-\epsilon}{2}\left(\frac{-\epsilon}{b-\epsilon}\right)^2-\epsilon\left(\frac{-\epsilon}{b-\epsilon}\right)\\
&=\frac{\epsilon^2}{2(b-\epsilon)}+\mathcal{O}(\epsilon^3),
\end{aligned}
\label{ejhncjdnvsd}
\end{eqnarray}
where the last equality is due to (\ref{erieief}). Solving (\ref{ejhncjdnvsd}) we get
$$x_u=\frac{\epsilon}{\sqrt{b-\epsilon}}+\mathcal{O}(\epsilon^2),\qquad x_s=-\frac{\epsilon}{\sqrt{b-\epsilon}}+\mathcal{O}(\epsilon^2)$$
for $0<\epsilon<\epsilon_2$ by $x_u>x_s$. Consequently, $P_L(x_1, \epsilon)$ is defined for  $x_s<x_1<x(\epsilon)$ and $0<\epsilon<\epsilon_2$. Moreover, $x(\epsilon)<P_L(x_1)<x_u$.

In conclusion, we take the definition domain of $P_L(x_1, \epsilon)$ as $\varpi_l(\epsilon)<x_1<x(\epsilon)$,
where
\begin{eqnarray}
\varpi_l(\epsilon)=
\left\{
\begin{aligned}
&-1\qquad &&{\rm when~ }E_Y~ {\rm is~ of~ focus~type, ~ i.e.,}~ b-\epsilon<0, \\
&x_s=-\frac{\epsilon}{\sqrt{b-\epsilon}}+\mathcal{O}(\epsilon^2) \qquad &&{\rm when~} E_Y~ {\rm is~ a~ saddle,~ i.e.,}~ b-\epsilon>0.
\end{aligned}
\right.
\label{jfeuihuncjsdh}
\end{eqnarray}

{\it Step 3. The full Poincar\'e map $P$.} Take $\epsilon_0=\min\{\epsilon_1, \epsilon_2\}$ and $\varpi(\epsilon)=|x(\epsilon)|$. In what follows $\epsilon_0>0$ can be reduced if necessary. Let
$$I(\epsilon)=\left(\varpi(\epsilon), \min\{-\varpi_l(\epsilon), \varpi_u(\epsilon)\}\right).$$
By (\ref{erieief}) and the definitions of $\varpi_l(\epsilon)$ and $\varpi_u(\epsilon)$
the interval $I(\epsilon)$ is non-empty for $0<\epsilon<\epsilon_0$.
According to the last two steps, we construct $P$ as the composition $P(x_0, \epsilon)=P_L(P_U(x_0, \epsilon), \epsilon)$ for $x_0\in I(\epsilon)$ and $0<\epsilon<\epsilon_0$.
Hence, a fixed point of $P(x_0, \epsilon)$ in the interval $I(\epsilon)$ corresponds to a crossing periodic orbit of $Z^f_\epsilon$.
Furthermore, from (\ref{saasmajf}), (\ref{PUM}) and (\ref{PLM}) the map $P(x_0, \epsilon)$ satisfies
$$\frac{1}{2}x_0^2+\epsilon f(-x_0, \epsilon)=\frac{1}{2}P(x_0, \epsilon)^2+\epsilon f(P(x_0, \epsilon), \epsilon),$$
i.e.,
\begin{eqnarray}
\frac{1}{2}x_0^2-\epsilon x_0\prod^m_{i=1}\left(x_0^2-\left(\frac{i\epsilon}{m}\right)^2\right)=
\frac{1}{2}P(x_0, \epsilon)^2+\epsilon P(x_0, \epsilon)\prod^m_{i=1}\left(P(x_0, \epsilon)^2-\left(\frac{i\epsilon}{m}\right)^2\right).
\label{ejfheuinjef}
\end{eqnarray}

{\it Step 4. Crossing limit cycles.} Now we study the crossing limit cycles of $Z^f_\epsilon$ using the Poincar\'e map $P$.
Since $0<|a|\le1/2$ and $0<|b|\le1/2$ as assumed in Proposition~\ref{ejhfjncjff}, we have $\min\{-\varpi_l(\epsilon), \varpi_u(\epsilon)\}>\sqrt 2\epsilon+\mathcal{O}(\epsilon^2)$ for $0<\epsilon<\epsilon_0$, so that $i\epsilon/m<\min\{-\varpi_l(\epsilon), \varpi_u(\epsilon)\}$ for all $i=1, 2, \cdot\cdot\cdot, m$ and $0<\epsilon<\epsilon_0$. On the other hand, it follows from (\ref{erieief}) that $i\epsilon/m>|x(\mu)|$, i.e., $i\epsilon/m>\varpi(\epsilon)$, for all $i=1, 2, \cdot\cdot\cdot, m$ and $0<\epsilon<\epsilon_0$. So $i\epsilon/m\in I(\epsilon)$
for all $i=1, 2, \cdot\cdot\cdot, m$ and $0<\epsilon<\epsilon_0$. Associate with (\ref{ejfheuinjef}), $x_0$ is a fixed point of $P$ in $I(\epsilon)$ if and only if $x_0=i\epsilon/m$, which implies that $Z^f_\epsilon$ has exactly $m$ isolated and nested crossing periodic orbits,
namely crossing limit cycles. Moreover, these crossing limit cycles intersect the positive $x$-axis at $(i\epsilon/m, 0)$, $i=1, 2, \cdot\cdot\cdot, m$.
Using the first integrals $H^f_X$ and $H^f_Y$, we get that the $m$ limit cycles obey the algebraic curves $\Gamma^+_i$ and $\Gamma^-_i$ defined in Proposition~\ref{ejhfjncjff}, $i=1, 2, \cdot\cdot\cdot, m$.

Finally, in order to determine the hyperbolicity and stability of $\Gamma_i$, $i=1, 2, \cdot\cdot\cdot, m$, taking the derivative with respect to $x_0$ for (\ref{ejfheuinjef}),
we have
$$
\begin{aligned}
\frac{dP}{dx_0}\left(\frac{i\epsilon}{m}\right)&=\frac{\frac{i\epsilon}{m}-2\epsilon\left(\frac{i\epsilon}{m}\right)^2\prod^m_{k=1, k\ne i}\left(\left(\frac{i\epsilon}{m}\right)^2-\left(\frac{k\epsilon}{m}\right)^2\right)}
{\frac{i\epsilon}{m}+2\epsilon\left(\frac{i\epsilon}{m}\right)^2\prod^m_{k=1, k\ne i}\left(\left(\frac{i\epsilon}{m}\right)^2-\left(\frac{k\epsilon}{m}\right)^2\right)}\\
&=\frac{1-2\epsilon^{2m}\left(\frac{i}{m}\right)\prod^m_{k=1, k\ne i}\left(\left(\frac{i}{m}\right)^2-\left(\frac{k}{m}\right)^2\right)}
{1+2\epsilon^{2m}\left(\frac{i}{m}\right)\prod^m_{k=1, k\ne i}\left(\left(\frac{i}{m}\right)^2-\left(\frac{k}{m}\right)^2\right)}.
\end{aligned}
$$
Thus $0<\frac{dP}{dx_0}\left(\frac{i\epsilon}{m}\right)<1$ (resp. $>1$) if $m-i$ is even (resp. odd), that is,
$\Gamma_i$ is hyperbolic and stable (resp. unstable) if $m-i$ is even (resp. odd). The proof of Proposition~\ref{ejhfjncjff}
is finished.
\end{proof}

\begin{proof}[{\bf Proof of Proposition~\ref{ejhadafjff}}]
Obviously, $Z^g_\epsilon=Z_0$ for $\epsilon=0$. The study of the bifurcated crossing limit cycles is extremely similar to the proof of Proposition~\ref{ejhfjncjff}. So we neglect some details. In fact, comparing the vector fields $Z^f_\epsilon=(X^f_\epsilon, Y^f_\epsilon)$ and $Z^g_\epsilon=(X^g_\epsilon, X^g_\epsilon)$,
we see $X^f_\epsilon=X^g_\epsilon$, so that we get the same upper Poincar\'e map
\begin{eqnarray}
P_U(x_0, \epsilon)=-x_0\qquad {\rm for}\quad 0<x_0<\varpi_u(\epsilon)~~{\rm and}~~0<\epsilon<\epsilon_1.
\label{PUMMMM}
\end{eqnarray}
as defined in (\ref{PUM}). Here $\varpi_u(\epsilon)$ is given in (\ref{rutireuiv}).
Besides, $Y^f_\epsilon$ and $Y^g_\epsilon$ have the same expression except that the function $f$ is replaced by $g$. With the replacement,
$O$ is an invisible fold point of $Y^g_\epsilon$, and $Y^g_\epsilon$ has $(0, -\epsilon/(b-\epsilon))$ as an equilibrium, which is of focus type if $b-\epsilon<0$ and a saddle if $b-\epsilon>0$. Therefore, carrying out a similar argument to Step 2 in the proof of Proposition~\ref{ejhfjncjff}, we can choose some $\epsilon_2>0$ and define a lower Poincar\'e map $P_L(x_1, \epsilon)$ for $\tilde\varpi_l(\epsilon)<x_1<0$ and $0<\epsilon<\epsilon_2$, where
$$
\tilde\varpi_l(\epsilon)=
\left\{
\begin{aligned}
&-1\qquad &&{\rm when~ }(0, -\epsilon/(b-\epsilon))~ {\rm is~ of~ focus~type, ~ i.e.,}~ b-\epsilon<0, \\
&\tilde x_s=-\frac{\epsilon}{\sqrt{b-\epsilon}}+\mathcal{O}(\epsilon^2) \qquad &&{\rm when~} (0, -\epsilon/(b-\epsilon))~ {\rm is~ a~ saddle,~ i.e.,}~ b-\epsilon>0,
\end{aligned}
\right.
$$
and $(\tilde x_s, 0)$ is the intersection between the stable manifold of $(0, -\epsilon/(b-\epsilon))$ and the negative $x$-axis. Notice that $\varpi_l(\epsilon)$ defined in (\ref{jfeuihuncjsdh}) and $\tilde\varpi_l(\epsilon)$ are the same in the sense of neglecting the higher order terms.
Since the first integral of $Y^g_\epsilon$ is
$$
H^g_Y(x, y)=\frac{1}{2}x^2+\epsilon g(x, \epsilon)-\frac{b-\epsilon}{2}y^2-\epsilon y,
$$
$P_L(x_1, \epsilon)$ satisfies
\begin{eqnarray}
\frac{1}{2}x_1^2+\epsilon g(x_1, \epsilon)=\frac{1}{2}P_L(x_1, \epsilon)^2+\epsilon g(P_L(x_1, \epsilon), \epsilon).
\label{PLLM}
\end{eqnarray}

The above analysis allows us to define a full Poincar\'e map $P(x_0, \epsilon)=P_L(P_U(x_0, \epsilon), \epsilon)$ for $x_0\in \tilde I(\epsilon)$ and $0<\epsilon<\epsilon_0$, where
$$\tilde I(\epsilon)=(0, \min\{-\tilde\varpi_l(\epsilon), \varpi_u(\epsilon)\}),\qquad \epsilon_0=\min\{\epsilon_1, \epsilon_2\}.$$
Hence, a fixed point of $P(x_0, \epsilon)$ in the interval $\tilde I(\epsilon)$
corresponds to a crossing periodic orbit of $Z^g_\epsilon$.
Furthermore, from (\ref{PUMMMM}) and (\ref{PLLM}) it follows that $P(x_0, \epsilon)$ satisfies
$$\frac{1}{2}x_0^2+\epsilon g(-x_0, \epsilon)=\frac{1}{2}P(x_0, \epsilon)^2+\epsilon g(P(x_0, \epsilon), \epsilon),$$
i.e.,
\begin{eqnarray}
\frac{1}{2}x_0^2=
\frac{1}{2}P(x_0, \epsilon)^2+\epsilon e^{-1/P(x_0, \epsilon)}\sin\left(\frac{\pi\epsilon}{P(x_0, \epsilon)}\right).
\label{eafaeinjef}
\end{eqnarray}

Now we study the fixed points of $P(x_0, \epsilon)$ in $\tilde I(\epsilon)$. Since $0<|a|\le1/2$ and $0<|b|\le1/2$ as assumed in Proposition~\ref{ejhadafjff}, $\min\{-\tilde\varpi_l(\epsilon), \varpi_u(\epsilon)\}>\sqrt2\epsilon+\mathcal{O}(\epsilon^2)$ for $0<\epsilon<\epsilon_0$, so that $\epsilon/i\in\tilde I(\epsilon)$ for $i\in\mathbb{N}^+$ and $0<\epsilon<\epsilon_0$. Here $\epsilon_0$ can be reduced if necessary. As a consequence, by (\ref{eafaeinjef}) we get that $x_0$ is a fixed point of $P(x_0, \epsilon)$ in $\tilde I(\epsilon)$ if and only if $x_0=\epsilon/i$, $i\in\mathbb{N}^+$. This means that $Z^g_\epsilon$ has infinitely many nested crossing limit cycles.
Moreover, these crossing limit cycles intersect the positive $x$-axis at $(\epsilon/i, 0)$, $i\in\mathbb{N}^+$.
Using the first integrals, we get that these crossing limit cycles obey the algebraic curves $\Theta^+_i$ and $\Theta^-_i$ defined in Proposition~\ref{ejhadafjff}, $i\in\mathbb{N}^+$.

Finally, computing the derivative of $P(x_0, \epsilon)$ with respect to $x_0$ for (\ref{eafaeinjef}), we get
$$
\begin{aligned}
\frac{dP}{dx_0}\left(\frac{\epsilon}{i}\right)&=\frac{\epsilon/i}{\epsilon/i-\pi i^2e^{-i/\epsilon}\cos(\pi i)}.
\end{aligned}
$$
So $0<\frac{dP}{dx_0}\left(\frac{\epsilon}{i}\right)<1$ (resp. $>1$) if $i$ is odd (resp. even), which implies that $\Theta_i$ is hyperbolic and stable (resp. unstable) if $i$ is odd (resp. even). This ends the proof of Proposition~\ref{ejhadafjff}.
\end{proof}

\end{document}